\documentclass[11pt]{amsart}
\usepackage{verbatim, amscd,epsfig,amsmath,amsthm,amstext,amssymb,
} 
\usepackage{graphics}
\usepackage[hyperindex=true,plainpages=false,colorlinks=false,pdfpagelabels]{hyperref}
\theoremstyle{plain} \newtheorem*{theorem*}{Theorem}
\newtheorem{theorem}{Theorem}[section] \newtheorem{lemma}[theorem]{Lemma}
\newtheorem{cor}[theorem]{Corollary}

\theoremstyle{definition}
\newtheorem{definition}[theorem]{Definition}

\theoremstyle{remark}
\newtheorem{rem}[theorem]{Remark}

\numberwithin{equation}{section}

\newcommand{\del}{\partial}
\newcommand{\delbar}{\bar{\del}}

\newcommand{\R}{\mathbb{R}}
\newcommand{\C}{\mathbb{ C}}
\newcommand{\Z}{\mathbb{ Z}}

\newcommand{\W}{\mathcal{W}}

\renewcommand{\H}{\mathbb{ H}}

\renewcommand{\P}{\mathbb{ P}}
\newcommand{\HP}{\H\P}

\newcommand{\dbar}{{\bar{\partial}}}
\newcommand{\invers}{^{-1}}
\DeclareMathOperator{\End}{End}
\DeclareMathOperator{\Hom}{Hom}

\DeclareMathOperator{\Harm}{Harm}

\DeclareMathOperator{\ord}{ord}

\DeclareMathOperator{\ev}{ev}
\DeclareMathOperator{\rank}{rank}
\DeclareMathOperator{\Spec}{Spec}
\newcommand{\Specs}{{\widetilde{\Spec}}}

\newcommand{\T}{{\mathcal T}}

\begin{document}
\title{Conformal maps from a $2$--torus to the $4$--sphere}
\date{\today}
\author{Christoph Bohle}
\author{Katrin Leschke}
\author{Franz Pedit}
\author{Ulrich Pinkall}

\address{Christoph Bohle\\
  Institut f\"ur Mathematik\\
  Technische Universit\"at Berlin\\
  Strasse des 17.~Juni 136\\
  10623 Berlin, Germany}

\address{Katrin Leschke\\
Department of Mathematics\\
 University of Leicester\\
 Leicester, LE1 7RH, UK 
}

\address{Franz Pedit\\
  Mathematisches Institut der  Universit{\"a}t T\"ubingen\\
  Auf der Morgenstelle 10\\
  72076 T\"ubingen\\
  Germany\\
  and \\
  Department of Mathematics \\
  University of Massachusetts\\
  Amherst, MA 01003, USA}

\address{Ulrich Pinkall\\
  Institut f\"ur Mathematik\\
  Technische Universit\"at Berlin\\
  Strasse des 17.~Juni 136\\
  10623 Berlin, Germany}


\thanks{First, third, and forth author supported by DFG SPP 1154 ``Global
  Differential Geometry''. Third author additionally supported by Alexander
  von Humboldt Foundation.}

\maketitle

\section{Introduction}\label{sec:introduction}

Much of differential geometry, as we know it today, has its origins in
the theory of surfaces. Initially, this theory had been local,
example oriented, and focused on special surface classes in $\R^3$.
This changed in the middle of the twentieth century when global
questions about surfaces in $\R^3$ were considered, e.g., do there
exists complete or compact surfaces in $\R^3$ with special
properties such as constant mean or Gaussian curvatures? The
integrability conditions for a surface in $\R^3$ are a system of
non--linear partial differential equations and those questions led to
significant advances in the global analysis of geometrically defined
partial differential equations. Aside from minimal surfaces, whose
analysis is essentially governed by the Cauchy--Riemann equation,
significant progress has been made in the description of (non--zero)
constant mean curvature surfaces and Willmore surfaces. These are
solutions to the elliptic variational problems for area under
constrained volume, and for the bending, or Willmore, energy $\int
H^2$ given by averaging the mean curvature square over the surface.

For constant mean curvature surfaces the combination of non--linear elliptic
analysis and methods from integrable system theory goes some way towards
describing the moduli spaces of such surfaces of a given topology. A
particularly fortuitous situation occurs for constant mean curvature tori:
those are the periodic orbits of an algebraically completely integrable system
whose phase space is the universal Jacobian over an appropriate moduli of
compact Riemann surfaces \cite{PS89,Hi,Bob}.

The analysis of Willmore surfaces turns out to be more complicated since the
integrability condition is a system of non--linear fourth order partial
differential equations. Results are less complete and details not all worked
out when compared to the constant mean curvature case. Nevertheless, Willmore
tori not M\"obius congruent to Euclidean minimal surfaces are still given by
an algebraically completely integrable system \cite{S02,cw_tori}.

That surfaces of genus one should play such a prominent role can
be seen in a number of ways. A constant mean curvature surface, and also a
Willmore surface, can be described by a circle family of flat
${\bf SU}(2)$ respectively ${\bf Sp}(2)$--connections.
Complexifying, one obtains a meromorphic  $\P^1$--family of ${\bf
Sl}(2,\C)$ respectively ${\bf Sp}(2,\C)$--connections with simple
poles at $0$ and $\infty$ in $\P^1$. On a  surface of genus one
the holonomy representation of this family of flat connections,
with respect to a chosen base point on the torus, is abelian. The
characteristic polynomial of any of the generators of the holonomy
thus gives rise to the same Riemann surface, the {\em spectral curve},
 which is a double, respectively
quadruple, branched cover over $\C_{*}=\P^1\setminus\{0,\infty\}$.
The ellipticity of the two variational problems allows to compactify
the spectral curve \cite{Hi,cw_tori} 
to a double, respectively quadruple, branched
cover of the Riemann sphere $\P^1$. The spectral curve has a real
structure induced by the complexification of the circle family to
a $\C_{*}$--family of flat connections. 

The eigenlines of the $\C_{*}$--family of holonomies extend to a holomorphic
line bundle, the {\em eigenline bundle}, over the compactified spectral curve.
Changing the base point on the torus, with respect to which the holonomy
representations were computed, the $\C_{*}$--family of holonomies gets
conjugated and thus the spectral curve is left unchanged.  What does change is
the holomorphic isomorphism type of the eigenline bundle due to contributions
from the simple poles at $0$ and $\infty$ of the $\C_{*}$--family of flat
connections. In fact, the eigenline bundles for the various base points on the
torus describe a real 2--dimensional subtorus in the Jacobian of the spectral
curve. It turns out that this subtorus is tangent to the real part of the Abel
image of the spectral curve inside the Jacobian at the origin.

What we have just described is known in the mathematical physics literature as
finite gap integration. The integrability conditions for constant mean
curvature and Willmore tori in $\R^3$, the zero--curvature equations of the
corresponding families of flat connections, are the elliptic affine Toda field
equations for the groups ${\bf SU}(2)$ --- the elliptic sinh--Gordon equation
--- and ${\bf Sp}(2)$ respectively \cite{S4Paper, DurhamBoys}. Such equations,
often referred to as soliton equations, are studied in mathematical physics
and are known to have an infinite hierarchy of commuting flows. The crucial
observation here is that at a doubly periodic solution those flows span a
finite dimensional space \cite{PS89, BFPP}, which turns out to be the Jacobian
of our previous description. The osculating flag to the Abel image of the
spectral curve inside the Jacobian at the origin describes the hierarchy of
flows modulo lower order flows. The importance of finite gap solutions, which
include doubly periodic solutions, comes from their explicit nature: these
solutions can be written down in terms of theta functions on the corresponding
spectral curve. As a consequence, one obtains theta function parameterizations
of constant mean curvature tori \cite{Bob} and, in principle, also of Willmore
tori.

Despite this rather complete picture of the moduli spaces of constant mean
curvature and Willmore tori a number of basic questions remain unanswered. For
instance, what are the minimal values of the two variational problems and on
which tori are they achieved? A well--known conjecture, first formulated by
Willmore, states that the minimal value of the bending energy $\int H^2$ over
immersed tori is $2\pi^2$ and is attained on the Clifford torus in the
$3$--sphere.  Note that the bending energy is invariant under the larger group
of M\"obius transformations of $\R^3$. Restricting to constant mean curvature
surfaces, which themselves are not M\"obius invariant, the bending energy
becomes the area functional and one deals with area estimates of constant mean
curvature tori. In this context one should also mention the conjecture, due to
Lawson, that the Clifford torus in $S^3$ is the only embedded minimal torus.
Using techniques from integrable system theory together with quaternionic
holomorphic geometry some progress has been made towards a resolution of these
conjectures \cite{S02, Klassiker,M&M}.

Instead of focusing on the ``classical" solutions of the above action
functionals on surfaces this paper shifts attention to the study of the moduli
space of all conformal immersions of a Riemann surface. This idea is
reminiscent to the path integral quantization in physics where one averages
over the space of all contending fields rather then just the classical
configurations given by the critical points of the action. Describing this
more fundamental space of all conformal immersions of a Riemann surface also
sheds new light on the constructions of constant mean curvature and Willmore
tori.

Even though one is traditionally interested in surfaces in $\R^3$, our
constructions necessitate to also consider surfaces in $\R^4$. The
conformality condition for an immersed surface is invariant under M\"obius
transformations and thus we are concerned with the space $\mathcal{M}$ of
(branched) conformal immersions $f\colon M\to S^4$ of a Riemann surface $M$
into the $4$--sphere $S^4$ up to M\"obius equivalence.  A special situation
arises for oriented branched conformal immersions of a compact Riemann surface
$M$ taking values in $S^2\subset S^4$: these are the meromorphic functions on
$M$ which can all be written down in terms of theta functions on $M$.
Projections of holomorphic curves in $\P^3$ via the twistor map to $S^4$
result in branched conformal immersions $f\colon M\to S^4$ into the
$4$--sphere which, in fact, are Willmore surfaces. These make up a portion of
the moduli space $\mathcal{M}$ described by the meromorphic functions on $M$.
Generally though, a conformal immersion $f\colon M\to S^4$ will not be of this
simple type and it is exactly those which are of interest in the present
paper.

To provide some intuition how we view the moduli space of conformal
immersions let us assume that $\mathcal{M}$ is a phase space of an integrable
system similar to the ones described above. If this were the case, we would
have to see abelian groups --- the energy shells of the integrable system ---
acting on the phase space. The least we should see though are Darboux
transformations: these are transformations on $\mathcal{M}$ which obey Bianchi
permutability and thus span an abelian group. In the afore mentioned example
of a constant mean curvature torus $f\colon T^2\to \R^3$ this abelian group is
the Jacobian of its spectral curve. The Darboux transforms are the secants to
the Abel image at the origin which span the Jacobian. In this way we can think
of the spectral curve, or rather its Abel image, as the space of all Darboux
transformations. This suggests that finding a notion of Darboux
transformations for a general conformal immersion $f\colon M\to S^4$ will
allow us to construct an analogue of the spectral curve.

Historically, there is the {\em classical} Darboux transformation of
isothermic surfaces \cite{darboux,udo_habil}: two conformally immersed
surfaces $f, f^{\sharp}\colon M\to S^4$ are classical Darboux transforms of
each other if there exists a $2$--sphere congruence $S$, that is, $S(p)$ is an
oriented round $2$--sphere in $S^4$, touching $f$ and $f^{\sharp}$ at $p\in
M$.  Darboux showed that if two surfaces are related by a classical Darboux
transform then both surfaces have to be isothermic, meaning that the principal
curvature lines of the surface are conformal coordinates on $M$.

In order to study general conformal immersions $f\colon M\to S^4$ we have to
relax the touching conditions which characterize Darboux transforms: instead
of demanding that the $2$--sphere congruence $S$ touches both surfaces $f$ and
$f^{\sharp}$, we only demand that $S$ touches $f$ and left--touches
$f^{\sharp}$ at the corresponding points. Two oriented planes through the
origin in $\R^4$ are left--touching, if their oriented intersection great
circles in $S^3\subset \R^4$ correspond under right multiplication by $S^3$.
Note that even if the original surface $f$ is contained in $3$--space, the
Darboux transform $f^{\sharp}$ has to lie in $4$--space to avoid both surfaces
to be isothermic: left--touching planes in $\R^3$ automatically coincide.
This explains the necessity to study conformal immersions into $4$--space,
even if one is mainly concerned about surfaces in $3$--space.  The Darboux
transformation so defined is invariant under M\"obius transformations and
satisfies Bianchi permutability: if $f^{\sharp}$ and $f^{\flat}$ are Darboux
transforms of $f\colon M\to S^4$ then there is a conformal map $\hat{f}\colon
M\to S^4$ which is a Darboux transform of $f^{\sharp}$ and $f^{\flat}$.

Following our intuitive reasoning, we view the space of all Darboux transforms
$f^{\sharp}\colon M\to S^4$ of a conformal immersion $f\colon M\to S^4$ as an
analog of the spectral curve. Our aim is to show that under certain
circumstances this space indeed has the structure of a Riemann surface.

We model the M\"obius geometry of the $4$--sphere by the projective
geometry of the quaternionic projective line $\HP^1$. A map
$f\colon M\to\HP^1$ is given by the quaternionic line subbundle
$L\subset V$ of the trivial $\H^2$--bundle $V$ with fiber
$L_p=f(p)\subset\H^2$ for $p\in M$. Conformality of $f$
is expressed by the existence of a (quaternionic) holomorphic
structure, that is, a first order elliptic operator
\begin{equation*}
\label{eq:i-holo}
 D\colon \Gamma(V/L)\to \Gamma(\bar{K}\,V/L)
\end{equation*}
on the line bundle $V/L$. The kernel $H^0(V/L)$ of $D$ defines the space of
(quaternionic) holomorphic sections of $V/L$ which contains the
$2$--dimensional linear system $H\subset H^0(V/L)$ obtained by projecting
constant sections in $V$ modulo $L$. This linear system gives rise to the
conformal immersion $f$ via the Kodaira correspondence \eqref{eq:Kodaira}.

In order to obtain an analytic description of Darboux transformations, we show
that a Darboux transform $f^{\sharp}$ of $f$ corresponds to a non--trivial
holomorphic section $\psi$ with monodromy of $V/L$, that is, a section $\psi$
on the universal cover of $M$ satisfying
\begin{equation*}
\label{eq:i-holo-mono} D\psi=0\quad\text{and}\quad
\gamma^{*}\psi=\psi h_{\gamma}
\end{equation*}
for some representation $h\colon \pi_1(M)\to \H_{*}$ of the fundamental
group of $M$. Therefore, to describe the space of Darboux
transforms of $f$, we have to characterize the
subspace of possible monodromies $h\in\Hom(\pi_1(M),\H_{*})$
up to conjugation by $\H_{*}$ of the holomorphic 
structure $D$ on $V/L$.

At this juncture we specialize to the situation when $M=T^2$ is a
$2$--torus and $f\colon T^2\to S^4$ is a conformal immersion whose
normal bundle has degree zero. Then the  fundamental group $\pi_1(T^2)$
is a rank $2$ lattice $\Gamma\subset \R^2$ and every
representation $h\in \Hom(\Gamma,\H_{*})$ can be conjugated to a
complex representation unique up to complex conjugation. The subspace
\[
\Spec(V/L,D)\subset \Hom(\Gamma,\C_{*})
\]
of possible complex monodromies of $D$ on $V/L$ over the torus $T^2$ is called
the {\em spectrum} of $D$. Here we make contact to the Floquet theory of the
periodic operator $D$ outlined by Taimanov \cite{Ta98}, Grinevich and Schmidt
\cite{GS98,S02} in their approach to the spectral curve.  Conjugating the
operator $D$ by the function $e^{\int\omega}$ on $\R^2$, where $\omega\in
\Harm(T^2,\C)$ is a harmonic $1$--form, we obtain a holomorphic family
\[
D_{\omega}\colon \Gamma(V/L)\to\Gamma(\bar{K}\,V/L)
\]
of first order elliptic operators parameterized over $ \Harm(T^2,\C)\cong
\C^2$. The kernel of $D_{\omega}$ describes precisely the holomorphic sections
of $D$ with monodromy $h=e^{\int\omega}$.  Since the normal bundle degree of
the immersion $f\colon T^2\to S^4$ is zero the family $D_{\omega}$ has index
zero. This implies \cite{ana} that the spectrum $\Spec(V/L,D)\subset
\Hom(\Gamma,\C_{*})$ is a $1$--dimensional analytic subvariety given by the
zero locus of the determinant $\det D_{\omega}$ of the holomorphic
family~$D_{\omega}$. Moreover, the kernel of $D_{\omega}$ along the spectrum
is generically $1$--dimensional.  The normalization
$h\colon\Sigma\to\Spec(V/L,D)$ of the spectrum to a Riemann surface is the
{\em spectral curve} $\Sigma$ of the immersed torus $f\colon T^2\to S^4$. It
carries a fixed point free real structure $\rho$ induced by complex
conjugation on the spectrum $\Spec(V/L,D)$.  The kernels $\ker D_{\omega}$
give rise to a complex holomorphic line bundle $\mathcal{L}$ over $\Sigma$
whose fiber $\mathcal{L}_x$ over $x\in\Sigma$ is spanned by a holomorphic
section $\psi$ of $V/L$ with monodromy $h(x)$. Such a section gives rise to a
Darboux transform $f^{\sharp}$ of $f$. Therefore $\Sigma/\rho$ parametrizes
the space of generic Darboux transforms of the conformal immersion $f\colon
T^2\to S^4$.

Analyzing the behavior of the holomorphic family $D_{\omega}$
for large monodromies shows \cite{ana} that the spectrum $\Spec(V/L,D)$ is
asymptotic to the {\em vacuum} spectrum $\Spec(V/L,\delbar)$ belonging to the
complex linear part $\delbar$ of $D$ on $V/L$. The vacuum spectrum
$\Spec(V/L,\delbar)\subset \Hom(\Gamma,\C_{*})$ is a singular curve
isomorphic to $\C\cup\C$ with the $\Z^2$--lattice of real
representations as double points. Therefore, the normalization of
the vacuum spectrum can be compactified to two copies of $\P^1$ by
adding two points at infinity and the real structure $\rho$
exchanges the two components.

To summarize, we have associated to a conformal immersion $f\colon T^2\to S^4$
from a $2$--torus to the $4$--sphere with zero normal bundle degree a Riemann
surface $\Sigma$ with either one or two ends, the spectral curve of $f$.
Depending on whether $\Sigma$ has one or two ends, the genus of $\Sigma$ is
infinite or finite. The spectral curve has a fixed point free real structure
$\rho$ induced by complex conjugation on the spectrum $\Spec(V/L,D)$.  The
kernels of $D_{\omega}$ define a complex holomorphic line bundle $\mathcal{L}$
over $\Sigma$ of quaternionic type, that is,
$\rho^*{\mathcal{L}}\cong\bar{\mathcal{L}}$. The fibers of $\mathcal{L}$ over
$x\in\Sigma$ describe Darboux transformations $f^{x}$ of $f$, whereby fibers
over $x$ and $\rho(x)$ give rise to the same Darboux transform. The resulting
map
\[
F\colon T^2\times {\Sigma} \to S^4\,,\quad F(p,x)=f^x(p)
\]
is conformal in the first factor
and lifts, via the twistor projection $\P^3\to S^4$, to a  map
\[
\hat{F}\colon T^2\times\Sigma\to\P^3
\]
which is holomorphic in the second factor.

As an example, we consider homogeneous tori $f\colon T^2\to S^4$ given by the
products of two circles of varying radii.  The holomorphic structure $D
=\delbar+Q$ on $V/L$ has constant $Q$ in an appropriate trivialization. The
spectrum $\Spec(V/L,D)$ is a perturbation of the vacuum spectrum
$\Spec(V/L,\delbar)$ by $Q$ which has the effect that one of the double points
of $\Spec(V/L,\delbar)$ resolves into a handle. Consequently, $\Spec(V/L,D)$
is connected and its normalization $\Sigma$ compactifies to the Riemann sphere
$\P^1$ by adding two points $o$ and $\infty=\rho(o)$ at infinity .  The map
$\hat{F}$ extends holomorphically into the points at infinity and is given by
a certain $T^2$--family of rational cubics.  Therefore, the map $F$, obtained
from $\hat{F}$ by twistor projection, yields a $T^2$--family of M\"obius
congruent Veronese embeddings of $\R\P^2=\P^1/\rho$ in $S^4$. The original
homogeneous torus $f$ is recovered by evaluating this family at infinity,
$f=F(-,o)=F(-,\infty)$.

In contrast to this example, the general conformal immersion $f\colon T^2\to
S^4$ of a $2$--torus of zero normal bundle degree will not have a spectral
curve $\Sigma$ which can be compactified: the holomorphic structure $D=\delbar
+Q$ on $V/L$ is a perturbation of the vacuum $\delbar$ by some ``potential"
$Q\in\Gamma(\bar{K}\End(V/L))$ which, based on physical intuition, has the
effect that the $\Z^2$--lattice of double points of the vacuum spectrum
$\Spec(V/L,\delbar)$ resolves into a $\Z^2$--lattice of handles for $\Sigma$
accumulating at infinity . Even though the theory of such Riemann surfaces of
infinite genus \cite{FKT03} resembles to some degree the theory of compact
Riemann surfaces, it is not algebro--geometric in nature.

This leads us to consider conformal tori $f\colon T^2\to S^4$ of zero normal
bundle degree whose spectral curves have finite genus and therefore two ends.
In this case only finitely many double points of the vacuum spectrum
$\Spec(V/L,\delbar)$ become handles for $\Spec(V/L,D)$. In the mathematical
physics literature the corresponding potentials $Q\in\Gamma(\bar{K}\End(V/L))$
with $D=\delbar +Q$ are called finite gap potentials. The normalization
$\Sigma$ of $\Spec(V/L,D)$ can therefore be compactified by the addition of
two points $o$ and $\infty=\rho(o)$ at infinity as in the example of
homogeneous tori above. The genus of the compactified spectral curve
$\bar{\Sigma}$ is called the {\em spectral genus} of the conformal torus
$f\colon T^2\to S^4$. Important examples of conformal tori of finite spectral
genus include, but are not confined to, constant mean curvature \cite{PS89,Hi}
and (constrained) Willmore tori \cite{S02,cw_tori}. In fact, there is
reasonable evidence from analogous situations in the literature \cite{igor,
  march} that conformal tori of finite spectral genus are dense in the space
of all conformal tori of zero normal bundle degree in the $4$--sphere.  When
$\Sigma$ has finite genus, we show that the map $\hat{F}\colon
T^2\times\Sigma\to\P^3$ extends holomorphically in the second component to
$T^2\times \bar\Sigma$.  The resulting $T^2$--family of algebraic curves in
$\P^3$, respectively their twistor projections to the $4$--sphere, recovers
the initial conformal torus $f\colon T^2\to S^4$ of finite spectral genus by
evaluation $f=F(-,o)=F(-,\infty)$ at infinity.

At this stage, we have come some way in describing conformal immersions of
$2$--tori into the $4$--sphere with zero normal bundle degree and finite
spectral genus in terms of algebro--geometric data: a compact finite genus
curve $\bar\Sigma$ with fixed point free real structure and a $T^2$--family of
algebraic curves into $\P^3$ compatible with the real structure.  As it turns
out \cite{ana} this $T^2$--family is given by linear motion on the real part
of the Jacobian of $\bar\Sigma$ tangent to the real part of its Abel image at
the origin.  In fact, this flow is the first in a hierarchy of commuting flows
preserving the Willmore energy coming from the osculating flag to the Abel
image at the origin.  It is precisely this observation which makes the space
of conformal maps $f\colon T^2\to S^4$ the phase space of a completely
integrable system, containing constant mean curvature and Willmore tori as
invariant subspaces.  This system is a geometric manifestation of the
Davey--Stewartson \cite{DS} hierarchy known to mathematical physicists and
what we have described is a geometric version of its finite gap integration
theory.

\section{The Darboux transformation}\label{sec:darboux}

We introduce the Darboux transformation for conformal immersions $f\colon M
\to S^4$ of a Riemann surface $M$ into the 4--sphere.  Similar to Darboux's
classical transformation of isothermic surfaces, our transformation of
conformal immersions is also given by a non--linear, M\"obius geometric
touching condition with respect to a $2$--sphere congruence.  Whereas the
classical Darboux transformation can be computed by solving a system of linear
ordinary differential equations, the general Darboux transformation of
conformal maps is described analytically by a quaternionic holomorphicity
condition, a linear first order elliptic partial differential equation.

\subsection{Conformal maps into $S^4$}
We model the M\"obius geometry of $S^4$ by the projective geometry of the
(quaternionic) projective line $\HP^1$. Thus, a map $f\colon M \to S^4$ is
given by the line subbundle $L\subset V$ of the trivial $\H^2$--bundle $V$
over $M$, where the fiber $L_p$ over $p\in M$ is the projective point $L_p =
f(p)$. In other words, $L =f^*\T$ is the pullback under $f$ of the
tautological line bundle $\T$ over $\HP^1$. Since the tangent bundle of
$\HP^1$ is $\Hom(\T,\underline{\H}^2/\T)$, the derivative of $f$ corresponds
to the 1--form
\[
\delta = \pi d|_L \in\Omega^1(\Hom(L,V/L))\,.
\]
Here $ d $ denotes the trivial connection on $V$ and $\pi\colon V \to V/L$
is the canonical projection.

To describe the conformality of $f$ note \cite{coimbra,PP} that over immersed
points $p\in M$ the real 2-plane $\delta(T_p M)\subset \Hom(L_p,\H^2/L_p)$ is
given by
\[
J_pX = X\tilde J_p\,, \quad  X\in \Hom(L_p,\H^2/L_p)
\]
for uniquely existing complex structures $J_p$ on $\H^2/L_p$ and $\tilde J_p$ on
$L_p$ which are compatible with the orientation on $\delta(TM)$ induced
from the Riemann surface $M$. 
Therefore, if $*$ denotes the complex structure on $TM^*$, the
conformality equation for $f$ over immersed points reads
\begin{equation}
\label{eq:conformal}
*\delta  = J\delta = \delta \tilde J\,.
\end{equation}
In what follows, we consider conformal maps $f\colon M \to S^4$ for which
at least one of the complex structures $J$ or $\tilde J$ extends smoothly
across the branch points.  This class of conformal maps, which
includes conformal immersions, can be described in terms of
quaternionic holomorphic geometry \cite{Klassiker}.  Notice that the
point--point duality of $\HP^1$ exchanges $L\subset V$ with
$L^\perp\subset V^*$.  Therefore, a conformal map $f$ with $*\delta
=J\delta$ becomes, via this duality, the conformal map $f^\perp$ with
$*\delta^\perp = \delta^\perp J^*$, where we identify $V/L =
(L^\perp)^*$ and $\delta^\perp = - \delta^*$.

An important invariant of a conformal map $f$ with
$*\delta = J \delta$ is its associated quaternionic holomorphic
structure (\ref{eq:hol_structure}) on the line bundle $V/L$. This
structure is given by the first order linear elliptic operator
\begin{equation}
\label{eq:hol_structure_V/L}
D\colon \Gamma(V/L) \to \Gamma(\bar{K}\,V/L) \quad \text{ defined by }\quad D \pi = (\pi  d )''\,,
\end{equation}
where $K$ denotes the canonical bundle of the Riemann surface $M$. The
operator $D$ is well--defined since $\pi d |_L=\delta\in\Gamma(K\Hom(L,V/L))$
and thus $(\pi d |_L)''=0$. If $f$ is non--constant the canonical projection
$\pi$ realizes $\H^2$ as a 2--dimensional linear system $H \subset H^0(V/L)$
whose Kodaira embedding \eqref{eq:Kodaira} is $L^\perp\subset V^*$.

The \emph{Willmore energy} of the conformal map $f$ is given by the
Willmore energy (\ref{eq:Willmore_energy_holbundle}) of the
holomorphic line bundle $V/L$, that is,
\begin{equation}
\label{eq:Willmore_energy}
\W(f)=\W(V/L,D) = 2\int_M <Q\wedge *Q>\,,
\end{equation}
where $Q\in\Gamma(\bar{K}\End_{-}(V/L))$ is the $J$--anticommuting part
of $D$.  

If $f\colon M \to S^4$ is an immersion also $L\subset V$ has a complex
structure $\tilde{J}$ and $*\delta=J\delta=\delta\tilde{J}$ by
\eqref{eq:conformal}. In particular, $\Hom_{+}(L,V/L)$ is the tangent bundle
and $N_f=\Hom_{-}(L,V/L)$ the normal bundle of $f$ where $\Hom_{\pm}(L,V/L)$
denote the complex linear, respectively complex antilinear, homomorphisms.
Since $\delta\in\Gamma(K\Hom_{+}(L,V/L))$ is a complex bundle isomorphism the
normal bundle degree of the conformal immersion $f\colon M\to S^4$, in case
$M$ is compact, calculates to
\begin{equation}
\label{eq:normal_degree}
\deg N_f=2\deg V/L + \deg K\,,
\end{equation}
where the degree of a quaternionic bundle is defined in \eqref{eq:degree}.  As
shown in \cite{coimbra}, Proposition~11, the M\"obius invariant $2$--form
$2<Q\wedge *Q>$, the integrand for the Willmore energy
\eqref{eq:Willmore_energy}, coincides over immersed points of $f$ with
\[
2<Q\wedge *Q>= ( |\mathcal{H}|^2 - \mathcal{K} - \mathcal{K}^\perp)d\mathcal{A}\,.
\]
Here we have chosen a point at infinity $L_0\in S^4$ and $\mathcal{H}$ is the
mean curvature vector, $\mathcal{K}$ the Gaussian curvature,
$\mathcal{K}^\perp$ the normal bundle curvature, and $d\mathcal{A}$ the
induced area of $f$ as a map into $\R^4=S^4\setminus\{L_0\}$.  Since
$V=L\oplus L_0$ and $V/L\cong L_0$ the trivial connection $d$ on $V$ restricts
to a flat connection $\nabla$ on $V/L$ for which $\nabla''=D=\delbar+Q$.  Let
$\nabla=\hat{\nabla}+A+Q$ be the decomposition of $\nabla$ into the
$J$--commuting part $\hat{\nabla}$ and the $J$--anticommuting part
$A+Q\in\Omega^{1}(\End(V/L))$, which we have further decomposed into type.  If
we denote by $\hat{R}$ the curvature of the complex connection $\hat{\nabla}$,
then flatness of $\nabla$ implies
\begin{equation}\label{eq:chern}
J\hat{R}=<A\wedge *A>-<Q\wedge *Q>\,.
\end{equation}
From \cite{coimbra}, Proposition~8, we see that
\[
2<A\wedge *A>=|\mathcal{H}|^2 d\mathcal{A}
\]
over immersed points of $f$. Thus, by \eqref{eq:chern} the $2$--form
$(\mathcal{K} +\mathcal{K}^\perp)d\mathcal{A}$ is the Chern form of the bundle
$V/L$, the classical Willmore integrand $|\mathcal{H}|^2 d\mathcal{A}$ extends
smoothly into the branch points of $f$, and
\begin{equation}
\label{eq:Willmore-classical}
\int_{M} |\mathcal{H}|^2  d\mathcal{A}= \W(f)+4\pi\deg(V/L)\,.
\end{equation}
\begin{lemma}
\label{l:embedded}
Let $f\colon M\to S^4$ be a non--constant conformal map with $*\delta=J\delta$
of a compact Riemann surface $M$ into $S^4$.  If the classical Willmore energy
of $f$ as a map into $\R^4$ satisfies
\[
\int_{M}|\mathcal{H}|^2 d\mathcal{A}< 8\pi\,,
\]
then $f$ is a conformal embedding with trivial normal bundle. Moreover, $\dim
H^0(V/L)=2$ which means that all holomorphic sections $\psi\in H^0(V/L)$ are
of the form $\psi=\pi(v)$ with $v\in\H^2$ where $\pi\colon V\to V/L$ denotes
the canonical projection.
\end{lemma}
\begin{proof}
  The proof is a repeated application of the Pl\"ucker formula
  \eqref{eq:Pluecker} for $1$--dimensional linear systems $H\subset H^0(V/L)$
  which, under the assumption of the lemma together with
  \eqref{eq:Willmore-classical}, satisfy
\[
\ord H<2\,.
\]
If $\dim H^0(V/L)>2$ then the linear map $H^0(V/L)\to \H^2/L_p\oplus \H^2/L_q$
given by evaluation of sections at $p\neq q\in M$ has at least a
$1$--dimensional kernel $H\subset H^0(V/L)$. This means that $H$ contains a
holomorphic section vanishing at $p$ and $q$ and thus $\ord H\geq 2$.

Next we show that $f$ is injective. If this is not the case there are distinct
points $p\neq q$ on $M$ with $L_p=L_q\subset \H^2$. Choosing any non--zero
$v\in L_p$ the corresponding non--trivial holomorphic section $\psi=\pi (v)$
of $V/L$ vanishes at $p$ and $q$ and we again have $\ord H\geq 2$.

To prove that $f$ is an immersion, we have to show that $\delta\in
\Gamma(K\Hom(L, V/L))$ has no zeros.  If $\delta$ had a zero at
$p\in M$, we construct a non--trivial holomorphic section $\psi\in
H^0(V/L)$ which vanishes to second order at $p\,$: let $\psi$ be the
holomorphic section of $V/L$ given by $\psi=\pi(v)$, where $v\in
L_p\subset \H^2$ is non--zero. Then $\psi$ has a zero at $p\in M$ and for any
$\alpha\in \Gamma(L^{\perp})$ we get
\[
d_p<\alpha,\psi>=d_p<\alpha,v>=<d_p\alpha,v>=<\pi^\perp d_p\alpha,v>=
<\delta^{\perp}_p(\alpha),v>=0\,,
\]
where we used that $\delta^{\perp}=-\delta^{*}$ and thus
$\delta^\perp_p=0$. This shows that $\psi\in H^0(V/L)$ vanishes to
second order at $p\in M$ and therefore $\ord H\geq 2$.

The normal bundle degree of $f$ is the self intersection number of $f(M)$
which is zero since $f$ is an embedding.
\end{proof}
\subsection{Darboux transforms}
An oriented round 2--sphere in $S^4=\HP^1$ is given by a linear map $S\colon
\H^2 \to \H^2$ which has $S^2 = -1$: points on the 2--sphere are the fixed
lines of $S$. The resulting line subbundle $L_S\subset V$ of the trivial
$\H^2$--bundle over the 2--sphere satisfies $SL_S = L_S$.  Thus, we have
complex structures on $L_S$ and $V/L_S$ and the conformality equation
\eqref{eq:conformal} of the embedded round sphere $S$ is $*\delta_S =
S\delta_S = \delta_S S$.

Given a Riemann surface $M$ a \emph{sphere congruence} assigns to each
point $p\in M$ an oriented round 2--sphere $S(p)$ in $S^4$. In other
words, a sphere congruence is a complex structure
$S\in\Gamma(\End(V))$ on the trivial $\H^2$--bundle $V$ over $M$.

Now let $f\colon M \to S^4$ be a conformal map with induced line bundle $L=
f^*\T\subset V$.  A sphere congruence $S\in\Gamma(\End(V))$
\emph{envelopes} $f$ if for all points $p\in M$ the spheres $S(p)$
pass through $f(p)$, and the oriented tangent spaces to $f$ coincide
with the oriented tangent spaces to the spheres $S(p)$ at $f(p)$ over
immersed points $p\in M$:
\begin{equation}
\label{eq:touch}
 SL = L \quad \text{ and } \quad *\delta = S\delta = \delta S\,.
\end{equation}
It is a classical result \cite{darboux,udo_habil} that two distinct conformal
immersions $f$ and $f^\sharp$ from the same Riemann surface $M$
which are both enveloped by the same sphere
congruence $S$ have to be isothermic surfaces. To overcome this restriction, we
need to relax the enveloping condition: two oriented planes through the
origin in $\R^4$ are \emph{left--touching}, respectively
\emph{right--touching}, if their associated oriented great circles on $S^3$
correspond via right, respectively left, translation in the group $S^3$. Hence,
we say that a sphere congruence $S$ \emph{left--envelopes}, respectively
\emph{right--envelopes}, $f$ if for all points $p\in M$ the spheres $S(p)$
pass through $f(p)$, and the oriented tangent spaces to $f$ are left--touching,
respectively right--touching, to the oriented tangent spaces to the spheres
$S(p)$ at $f(p)$ over immersed points $p\in M$:
\begin{equation}
\label{eq:half_touch}
 SL = L \quad \text{ and } \quad *\delta = S\delta\,,  \quad \text{
 respectively } \quad *\delta = \delta S\,.
\end{equation}

\begin{definition}
\label{def:darboux}
Let $M$ be a Riemann surface. A conformal map $f^\sharp\colon M \to S^4$ is
called a \emph{Darboux transform} of a conformal immersion $f\colon M \to
S^4$ if $f^\sharp(p)$ is distinct from $f(p)$ at all points $p\in M$, and if
there exists a sphere congruence $S$ which envelopes $f$ and
left--envelopes~$f^\sharp$:
\begin{gather}
\label{eq:darboux}
V = L \oplus L^\sharp, \quad
SL = L,\quad *\delta  = S\delta = \delta S,  \quad \text{ and } \quad SL^\sharp = L^\sharp,\quad *\delta^\sharp  = S\delta^\sharp\,.
\end{gather}
\end{definition}
In particular, if $f$ and $f^\sharp$ are conformal immersions into
$S^3$ and $f^\sharp$ is a Darboux transform of $f$, then both $f$ and
$f^\sharp$ are isothermic. This follows from the fact that in 3--space
a half--enveloping sphere congruence is always enveloping.

There are a number of equivalent characterizations of Darboux transforms
$f^\sharp\colon M \to S^4$ of a conformal immersion $f\colon M\to S^4$,
including a description in terms of flat adapted connections and, more
generally, holomorphic sections with monodromy of the bundle $V/L$.  It is
this last analytic characterization which will play a fundamental role in our
construction of the spectral curve.

Let $f, f^{\sharp}\colon M\to S^4$ be maps distinct from each other at all
points $p\in M$.  Then $V = L \oplus L^\sharp$ and $\pi$ identifies $V/L=
L^\sharp$ whereas $\pi^\sharp$ identifies $V/L^\sharp=L$. The trivial
connection $d$ on $V$ decomposes as
\begin{equation}
\label{eq:splitted_connection}
 d  =
 \begin{pmatrix} \nabla^L& \delta^\sharp\\
                   \delta& \nabla^\sharp
\end{pmatrix}
\end{equation}
and flatness of $d$ implies that
\begin{equation}
\label{eq:d_flat}
d\delta=0, \quad  d\delta^\sharp =0,  \quad \text{ and } \quad R^\sharp = -
\delta\wedge\delta^\sharp\,,
\end{equation}
where $R^\sharp$ denotes the curvature of $\nabla^\sharp$.
\begin{lemma}
\label{l:darboux_equivalent}
Let $f\colon M \to S^4$ be a conformal immersion with
$*\delta=J\delta=\delta\tilde{J}$ and $f^\sharp\colon M \to S^4$ a map so that
$V = L \oplus L^\sharp$. Then we have the following equivalent
characterizations of Darboux transforms:
\begin{enumerate}
\item The map $f^\sharp$ is a Darboux transform of $f$.
\item The map $f^\sharp$ is conformal with $*\delta^\sharp=\tilde
  J\delta^\sharp$.
\item The connection $\nabla^\sharp$ on $L^\sharp$ induced by the
  splitting $V = L \oplus L^\sharp$ is flat.
\item There is a non--trivial section
  $\psi^\sharp\in\Gamma(\widetilde{L^\sharp})$ with monodromy satisfying
  $d \psi^\sharp\in\Omega^1(\tilde L)$.
\end{enumerate}
\end{lemma}
A section $\psi$ with monodromy of a vector bundle $W\to M$ is
a section of the pull--back bundle $\tilde{W}\to \tilde M$ of
$W$ to the universal cover $\tilde{M}\to M$ with 
\[
 \gamma^*\psi=\psi h_\gamma\,,
 \] 
 where $h\colon \pi_1(M)\to \H_*$ is a 
representation and $\gamma\in\pi_1(M)$ acts as a deck transformation.
\begin{proof}
  The unique sphere congruence $S$ touching $f$ and containing $f^{\sharp}$,
  expressed in the splitting $V=L\oplus L^{\sharp}$, is
\[
\begin{pmatrix} \tilde{J}&0\\0&J\end{pmatrix}\,.
\]
Therefore, the sphere congruence $S$ left--touches $f^{\sharp}$ if and only if
$*\delta^\sharp=\tilde J\delta^\sharp$, which proves the first equivalence by
\eqref{eq:darboux}.  The second equivalence is a direct consequence of
$R^\sharp = - \delta\wedge\delta^\sharp$ in \eqref{eq:d_flat} together with a
type consideration.  The last equivalence follows from
\eqref{eq:splitted_connection} because flatness of $\nabla^{\sharp}$ is
equivalent to the existence of parallel sections with monodromy of the line
bundle $\tilde{L}^{\sharp}$.
\end{proof}
The isomorphism $\pi\colon L^{\sharp}\to V/L$ pushes forward the connection
$\nabla^{\sharp}$ to a connection $\nabla$ on $V/L$ satisfying
\begin{equation}
  \label{eq:splitting_connection}
  \nabla\pi|_{\Gamma(L^\sharp)} = \pi\nabla^{\sharp}=\pi  d|_{\Gamma(L^\sharp)}.
\end{equation}
By construction the connection $\nabla$ is {\em adapted} to the complex
structure $D$ on $V/L$ defined in \eqref{eq:hol_structure_V/L}, that is,
$\nabla'' = D$.

For a fixed immersion $f$ the spaces of splittings $V = L \oplus L^\sharp$ and
the space of adapted connections on $V/L$ are affine spaces modelled on the
vector spaces $\Hom(V/L,L)$ and $\Gamma(K\End(V/L))$, respectively.  The map
assigning to a splitting the induced adapted connection
\eqref{eq:splitting_connection} is affine with underlying vector space
homomorphism
\begin{equation}
  \label{eq:homomorphism}
  \Gamma(\Hom(V/L, L)) \to \Gamma(\Gamma(K\End(V/L))) \qquad R \mapsto \delta R.
  \end{equation}

Since $f$ is an immersion and thus $\delta$ an isomorphism, the correspondence
assigning to a splitting the induced adapted connection on $V/L$ is an affine
isomorphism.  Together with Lemma~\ref{l:darboux_equivalent} we get a
characterization of Darboux transforms in terms of flat adapted connections on
$V/L$.
\begin{cor}
\label{c:flat_adapted}
Let $f\colon M\to S^4$ be a conformal immersion. Then there is a bijective
correspondence between the space of Darboux transforms $f^\sharp\colon M \to
S^4$ of $f$ and the space of flat adapted connections on $V/L$.  This
correspondence is given by the restriction to the space of Darboux transforms
of the affine isomorphism \eqref{eq:splitting_connection} assigning to a
splitting $V = L \oplus L^\sharp$ the adapted connection $\nabla$ on $V/L$.
\end{cor}
In what follows it is necessary to compute Darboux transforms from flat
adapted connections by using prolongations of holomorphic sections on $V/L$.
\begin{lemma}
\label{l:prolong}
Let $f\colon M\to S^4$ be a conformal immersion. Then the canonical projection
$\pi\colon V \rightarrow V/L$ induces a bijective correspondence between
sections $\hat{\psi}\in\Gamma(V)$ of $V$ satisfying
$d\hat{\psi}\in\Omega^1(L)$ and holomorphic sections $\psi \in H^0(V/L)$.
\end{lemma}\begin{proof}
  Let $\hat\psi_0\in\Gamma(V)$ be a lift of $\psi\in H^0(V/L)$, that is, $\pi
  \hat\psi_0 = \psi$. Then \eqref{eq:hol_structure_V/L} implies $ (\pi d
  \hat\psi_0)'' = D \pi \hat \psi_0 = D\psi = 0 $ and, since $f$ is an
  immersion, there is a unique section $\varphi\in \Gamma(L)$ with
  $\delta\varphi = (\pi d \hat\psi_0)$. But then $\hat \psi= \hat \psi_0
  -\varphi$ is the unique section of $V$ with the required properties.
\end{proof}
\begin{definition}
  Let $f\colon M\to S^4$ a conformal immersion. The {\em prolongation} of a
  holomorphic section $\psi\in H^0(V/L)$ is the unique section
  $\hat{\psi}\in\Gamma(V)$ with $\pi\hat{\psi}=\psi$ satisfying
  $d\hat{\psi}\in\Omega^1(L)$, that is, $\pi d\hat{\psi}=0$.
\end{definition}
Given a flat adapted connection $\nabla$ on $V/L$, we want to compute the
corresponding Darboux transform of $f$ from Corollary~\ref{c:flat_adapted}. We
take a parallel section $\psi\in\Gamma(\widetilde{V/L})$ over the universal
cover $\tilde{M}$ of $M$. Since $V/L$ is a line bundle this section has
monodromy.  But $\nabla$ is adapted so that $\psi\in H^0(\widetilde{V/L})$ is
also a holomorphic section with monodromy, that is,
 \begin{equation}
\label{eq:hol_mon}
D\psi = 0\quad\text{and}\quad \gamma^*\psi=\psi h_\gamma\,,
\end{equation}
for a representation $h\colon \pi_1(M)\to \H_*$.  Then the prolongation
$\hat{\psi}\in \Gamma(\tilde{V})$ is a section with the same monodromy $h$.
Moreover, as a parallel section $\psi$ has no zeros and neither
does~$\hat{\psi}$. This shows that the line bundle
\[
L^{\sharp}=\hat{\psi}\H\subset V
\]
is well defined over $M$ and satisfies $V=L\oplus L^{\sharp}$. As a
prolongation $\hat{\psi}$ has $d\hat{\psi}\in\Omega^1(\tilde{L})$ which is one
of the equivalent characterizations in Lemma~\ref{l:darboux_equivalent} for
Darboux transforms.  This shows that the map $f^\sharp\colon M\to S^4$
corresponding to the line bundle $L^\sharp$ is the Darboux transform of $f$
belonging to the adapted connection $\nabla$.

The parallel sections of flat adapted connections on $V/L$ are precisely the
holomorphic sections with monodromy of $V/L$ that are nowhere vanishing.  If a
holomorphic section $\psi\in\Gamma(\widetilde{V/L})$ has zeroes these are
isolated \cite{Klassiker} and, away from the finite set of zeros, there is a
unique flat adapted connection on $V/L$ which makes $\psi$ parallel.  Hence, a
holomorphic section $\psi$ of $V/L$ with monodromy gives rise to a
Darboux transform $f^\sharp$ of $f$ defined away from the zero locus of
$\psi$ by $L^\sharp = \hat \psi\H$, where $\hat \psi$ is the prolongation of
$\psi$. It follows from \cite{Klassiker}, Lemma~3.9,  that $f^\sharp$ extends 
continuously across the
zeros of $\psi$ where it agrees with~$f$.  We call such $f^\sharp$
\emph{singular} Darboux transforms of $f$.

\begin{lemma}
\label{l:singular_darboux}
Let $f\colon M\to S^4$ be a conformal immersion. Then there is a bijective
correspondence between the space of (singular) Darboux transforms
$f^\sharp\colon M \to S^4$ of $f$ and the space of non--trivial holomorphic
sections with monodromy up to scale of $ V/L$.  Under this correspondence,
non--singular Darboux transforms get mapped up to scale to nowhere vanishing
holomorphic sections with monodromy of $V/L$.
\end{lemma}

In contrast to Definition~\ref{def:darboux}, Lemma~\ref{l:darboux_equivalent},
and Corollary~\ref{c:flat_adapted} which characterize Darboux transforms by
non--linear M\"obius geometric or zero curvature conditions,
Lemma~\ref{l:singular_darboux} characterizes Darboux transforms in terms of
solutions to a linear elliptic equation.  Therefore, locally $V/L$ has an
infinite dimensional space of holomorphic sections without zeros and we get an
infinite dimensional space of local Darboux transforms of $f$. The situation
is rather different when considering global Darboux transforms $f^\sharp\colon
M \to S^4$ where $M$ has non--trivial topology. We shall see in the next
section that when $f\colon T^2\to S^4$ is a conformally immersed 2--torus of
zero normal bundle degree there is a Riemann surface worth of Darboux
transforms of $f$.  There are always the trivial Darboux transforms, the
constant maps $f^\sharp$, coming from holomorphic sections (without monodromy)
of the $2$--dimensional linear system $H \subset H^0(V/L)$.  Whenever such a
trivial Darboux transform $f^\sharp=f(p)$ is a point contained in the image of
$f$, we have a singular Darboux transform and at $p\in M$ the sphere
congruence $S$ degenerates to a point.

\subsection{Bianchi permutability}
An important feature of the classical Darboux transformation of isothermic
surfaces is the following permutability property \cite{Bia,udo_habil}: if
$f^\sharp$ and $f^\flat$ are Darboux transforms of an isothermic immersion $f$
then there exists an isothermic immersion $\hat f$ which simultaneously is a
Darboux transform of $f^\sharp$ and $f^\flat$. This property carries over to
general Darboux transforms of conformal immersions and will be used in
Section~\ref{sec:spectral} to show that the Darboux transformation is {\em
  isospectral}.
\begin{theorem}
\label{thm:bianchi}
Let $f\colon M\to S^4$ be a conformal immersion and let $f^{\sharp},
f^{\flat}\colon M\to S^4$ be two immersed Darboux transforms of $f$ so that
$f^\sharp$ and $f^\flat$ are distinct for all points in~$M$.  Then there
exists a conformal map $\hat f\colon M\to S^4$ which is a Darboux transform of
$f^\sharp$ and~$f^\flat$.
\end{theorem}
\begin{proof}
From Lemma~\ref{l:darboux_equivalent} we know that there exist
non--trivial sections $\psi^\sharp\in\Gamma(\widetilde{L^\sharp})$ and
$\psi^\flat \in\Gamma(\widetilde{L^\flat})$ with monodromies
$h^\sharp, h^\flat\colon \pi_1(M) \to \H_*$ satisfying $\pi
d\psi^\sharp = \pi d\psi^\flat =0$. 
Using the splitting $V = L \oplus L^\sharp$ we obtain
\[
d\psi^\sharp = \nabla^\sharp\psi^\sharp + \delta^\sharp\psi^\sharp =
\delta^\sharp\psi^\sharp \in\Gamma(\widetilde {KL})\,,
\]
where we again used Lemma~\ref{l:darboux_equivalent} to see that
$*\delta^\sharp= \tilde J \delta^\sharp$.  In particular, as a
parallel section $\psi^\sharp$ is nowhere vanishing and, since
$f^\sharp$ is an immersion, also $d \psi^\sharp
=\delta^\sharp\psi^\sharp$ is nowhere vanishing.
The same argument applied to $\psi^\flat$ yields the nowhere vanishing
sections $d\psi^\sharp$ and $d\psi^\flat$ with monodromies $h^\sharp$
and $h^\flat$  of the line bundle $ KL$.
Therefore, 
\begin{equation}
\label{eq:bianchi}
d\psi^\flat = d
\psi^\sharp\chi,
\end{equation}
where $\chi\colon \tilde M \to \H_*$ and $\gamma^*\chi= (h^\sharp)\invers \chi
h^\flat$. Since $V = L^\sharp \oplus L^\flat$ the section
\begin{equation}
\label{eq:double_darboux}
\varphi = \psi^\flat- \psi^\sharp \chi \in\Gamma(\tilde V)
\end{equation}
with monodromy $h^\flat$ has a nowhere vanishing projection to $V/L^{\sharp}$.
Then the line bundle $\hat L\subset V$ spanned by the nowhere vanishing
section $\varphi$ is a Darboux transform $\hat f\colon M \to S^4$ of
$f^\sharp$: by Lemma~\ref{l:darboux_equivalent} it is sufficient to show that
$\pi^\sharp d\varphi=0$ which follows immediately from
\eqref{eq:double_darboux} and (\ref{eq:bianchi}).  On the other hand, the
nowhere vanishing section $\varphi \chi\invers\in\Gamma(\tilde{\hat L})$ has
monodromy $h^\sharp$ and exhibits $\hat f$ as a Darboux transform of
$f^\flat$.
\end{proof}
\begin{rem}\label{rem:bianchi}
  The proof of the previous theorem shows the following: given a conformal
  immersion $f\colon M\to S^4$ and an immersed Darboux transform
  $f^{\sharp}\colon M\to S^4$ of $f$ then \eqref{eq:double_darboux} defines a
  monodromy preserving map between holomorphic sections with monodromy of
  $V/L$ and $V/L^{\sharp}$. 
\end{rem}

\subsection{The Willmore energy of Darboux transforms}
For immersions of compact surfaces the Darboux transform preserves the
Willmore energy up to topological quantities. In particular, for conformally
immersed tori with trivial normal bundle the Darboux transform preserves the
Willmore energy.
\begin{lemma}
\label{l:willmore_darboux}
Let $f\colon M\to S^4$ be a conformal immersion of a compact Riemann surface and
let $f^\sharp\colon M \to S^4$ be a Darboux transform of $f$. Then
\begin{equation*}
\W(f^\sharp)= \W(f) + 2\pi(\deg N_f -\deg K)\,,
\end{equation*}
where $N_f = \Hom_-(L,V/L)$ is the normal bundle of $f$.  
\end{lemma}
\begin{proof}
We first note that $ \delta \in
\Gamma(\Hom_+(L,K \, V/L))$ is a holomorphic bundle isomorphism: the
holomorphic structure $(\nabla^L)''$ on $L$ comes from the splitting
(\ref{eq:splitted_connection}) and $K\, V/L$ has the holomorphic
structure $d^\nabla$, where $\nabla$ is the flat adapted connection on
$V/L$ given in Corollary~\ref{c:flat_adapted}.  If $\varphi\in H^0(L)$
then (\ref{eq:d_flat}) implies
\[
d^\nabla\delta\varphi = (d \delta)\varphi - \delta\wedge
\nabla^L\varphi =- \delta\wedge (\nabla^L)''\varphi = 0\,,
\]
i.e., $\delta\varphi\in H^0(K\, V/L)$, which shows that $\delta$ is a
holomorphic bundle map. Since $f$ is an immersion $\delta$ is a
holomorphic isomorphism and thus 
$\W(L)=\W(K\, V/L)$. 
On the other
hand, $\pi^\sharp$ identifies the quaternionic holomorphic bundle
$V/L^\sharp$ with $L$ so that
\begin{equation}
\label{eq:willmore_rel}
\W(V/L^\sharp)=\W(L)=\W(K\, V/L)\,.
\end{equation}
The adapted flat connection $\nabla$ of $V/L$ can be decomposed into $J$
commuting and anti\-commuting parts $ \nabla = \hat\nabla +A+Q$ where
$\hat\nabla = \partial + \delbar$ is a complex connection and
$A\in\Gamma(K\End_-(V/L))$ and $Q\in\Gamma(K\End_-(V/L))$ are endomorphism
valued $1$--forms of type $K$ and $\bar{K}$ respectively. Since $\nabla$ is
adapted, $D = \nabla'' = \delbar + Q$, the Willmore
energy~(\ref{eq:Willmore_energy}) of $V/L$ is given by
\begin{equation}
\label{eq:willmore_Q}
\W(V/L)=2\int_M<Q\wedge *Q>\,.
\end{equation}
Decomposing \eqref{eq:hol_structure_decomp} the holomorphic structure
$d^\nabla$ on $K\, V/L$ into $J$ commuting and anticommuting parts
$d^\nabla = \delbar + \tilde Q$, one checks that
\[
\tilde Q \omega = A\wedge \omega
\]
for $ \omega\in \Gamma(K\, V/L)$. Therefore the Willmore energy of
$K\, V/L$ is given by
\begin{equation}
\label{eq:willmore_KV/L}
\W( K\, V/L) = 2\int_M <A\wedge *A>\,.
\end{equation}
The flatness of $\nabla=\hat{\nabla}+A+Q$ implies
\[
J\hat R = <A \wedge *A> - <Q \wedge *Q>\,,
\]
and (\ref{eq:willmore_Q}), (\ref{eq:willmore_KV/L}) yield
\[
 4\pi\deg(V/L) =  \W(K\, V/L) - \W(V/L)\,.
\]
The lemma now follows from (\ref{eq:willmore_rel}) together with the formula
\eqref{eq:normal_degree} for the normal bundle degree of $f$.
\end{proof}

\subsection{Conformal tori with $\int_{T^2}|\mathcal{H}|^2 d\mathcal{A}<8\pi$}
We know from Lemma~\ref{l:embedded} that the sublevel set
$\int_{T^2}|\mathcal{H}|^2 d\mathcal{A} <8\pi$ in the space of conformal tori
$f\colon T^2 \to S^4$ in the $4$--sphere consists of conformal embeddings with
trivial normal bundles.  From \eqref{eq:normal_degree} we see that $\deg
V/L=0$ which implies $\mathcal{W}(f)=\int_{T^2}|\mathcal{H}|^2 d\mathcal{A}$
by \eqref{eq:Willmore-classical}.  Every non--constant Darboux transform
$f^\sharp\colon T^2\to S^4$ of $f$ is again a conformal embedding and we can
apply Bianchi permutability repeatedly for conformal tori in this sublevel
set.
\begin{lemma}
\label{l:emb_Darboux}
Let $f\colon T^2\to S^4$ be a conformal map with $\int_{T^2}|\mathcal{H}|^2
d\mathcal{A}<8\pi$.  Then
\begin{enumerate}
\item $f$ is a conformal embedding and
  $\mathcal{W}(f)=\int_{T^2}|\mathcal{H}|^2 d\mathcal{A}$,
\item every non--constant Darboux transform $f^\sharp\colon T^2\to S^4$ is a
  conformal embedding and $\int_{T^2}|\mathcal{H}^\sharp|^2
  d\mathcal{A^\sharp}= \int_{T^2}|\mathcal{H}|^2 d\mathcal{A}$,
\item two distinct non--constant Darboux transforms $f^{\sharp}$ and $
  f^{\flat}$ of $f$ satisfy $f^{\sharp}(p)\neq f^{\flat}(p)$ for all $p\in
  T^2$, and
\item the only singular Darboux transforms are the constant maps contained in
  the image of $f$.
\end{enumerate}
\end{lemma}
\begin{proof}
  From Lemma \ref{l:singular_darboux} we see that a non-constant Darboux
  transform $f^\sharp$ is singular if and only if its corresponding
  holomorphic section $\psi\in H^0(\widetilde{V/L})$ with monodromy has a zero
  at, say, $p\in T^2$.  Since $\dim H^0(V/L) \geq 2$ there always is a
  holomorphic section $\varphi \in H^0(V/L)$ vanishing at $p\in T^2$.
  Therefore, the 2--dimensional linear system $H \subset H^0(\widetilde{V/L})$
  with monodromy (\ref{eq:system_monodromy}), (\ref{eq:ord_H}) spanned by
  $\psi$ and $\varphi$ has $\ord H\ge 2$.  But then the Pl\"ucker formula
  \eqref{eq:Pluecker} together with $\deg V/L=0$ gives the contradiction
  $\W(f)=\W(V/L)\geq 8\pi$.

  Therefore all non--constant Darboux transform $f^\sharp$ are non--singular
  so that $V=L\oplus L^\sharp$.  Since $f$ is an embedding,
  $*\delta=J\delta=\delta\tilde{J}$ where $\delta\in \Gamma(K\Hom_{+}(L,V/L))$
  is an isomorphism, which shows that the line bundle $L$ with complex
  structure $\tilde{J}$ has $\deg L=0$. On the other hand,
  Lemma~\ref{l:darboux_equivalent} says that the Darboux transform $f^\sharp$
  has $*\delta^\sharp=\tilde{J}\delta^\sharp$ which implies that the line
  bundle $V/L^\sharp=L$ with complex structure $\tilde{J}$ has $\deg
  V/L^\sharp=0$.  Applying \eqref{eq:Willmore-classical} and
  Lemma~\ref{l:willmore_darboux}, we therefore obtain
\[
\int_{T^2}|\mathcal{H}^\sharp|^2
d\mathcal{A^\sharp}=\mathcal{W}(f^\sharp)=\mathcal{W}(f)=\int_{T^2}|\mathcal{H}|^2
d\mathcal{A}<8\pi
\]
which shows that also $f^\sharp$ is a conformal embedding by
Lemma~\ref{l:embedded}.

If two non-constant Darboux transforms $f^{\sharp}$ and $f^{\flat}$ of $f$ had
a point in common, the section $\varphi$ in \eqref{eq:double_darboux} would
project under $\pi^{\sharp}$ to a holomorphic section with monodromy of
$V/L^{\sharp}$ with a zero.  Thus, the common non-constant Darboux transform
$\hat{f}$ of $f^{\sharp}$ and $f^{\flat}$ would be singular.

Finally, if $f^\sharp$ is constant then $V=L\oplus L^\sharp$ if and only if
$f^\sharp$ is not in the image of $f$.
\end{proof}

\section{The spectral curve}\label{sec:spectral}
In this section we show that a conformally immersed torus $f\colon T^2\to S^4$
of zero normal bundle degree gives rise to a Riemann surface, the spectral
curve $\Sigma$ of $f$. Our approach is geometric because the spectral curve is
introduced as the space parametrizing Darboux transforms of $f$. This point of
view yields a natural $T^2$--family of {\em holomorphic} maps of $\Sigma$ into
$\P^3$ from which the conformal immersion $f\colon T^2\to S^4$ can be
reconstructed.
\subsection{The spectrum of a holomorphic line bundle over a torus} 
We recall that Lemma \ref{l:singular_darboux} characterizes a Darboux
transform $f^\sharp\colon M\to S^4$ of a conformal immersion $f\colon M\to
S^4$ by a non--trivial holomorphic section $\psi\in \Gamma(\widetilde{V/L})$
with monodromy $h\colon\pi_1(M)\to\H_*$.  Scaling this holomorphic section
conjugates the representation $h$ but does not affect the Darboux transform
$f^{\sharp}$. This suggests to describe the parameter space of Darboux
transforms analytically in terms of the possible monodromies $h\colon
\pi_1(M)\to \H_*$ of non--trivial holomorphic sections of $\widetilde{V/L}$ up
to conjugation.  
\begin{definition}
\label{def:spec}
Let $W$ be a quaternionic line bundle with holomorphic structure $D$
over a Riemann surface $M$.  The {\em quaternionic spectrum} of $W$ is the
subspace
\begin{equation*}
\Spec_\H(W,D) \subset \Hom(\pi_1(M),\H_*)/\H_*
\end{equation*}
of conjugacy classes of possible monodromy representations $h\colon
\pi_1(M)\to \H_*$ for holomorphic sections of $\tilde{W}$. In other words, $h$
represents a point in $\Spec_\H(W,D)$ if and only if there exists a
non--trivial section $\psi\in \Gamma(\tilde{W})$ with
\[
D\psi = 0\, \quad\text{and}\quad \gamma^*\psi=\psi h_\gamma\,,
\]
where $\gamma\in\pi_1(M)$ acts by deck transformations.
\end{definition}
Applying this notion to the holomorphic line bundle $V/L$ induced by a
conformal immersion $f\colon M \to S^4$, we obtain from
Lemma~\ref{l:singular_darboux} a surjective map from the space of
(singular) Darboux transforms of $f$ onto the quaternionic spectrum
$\Spec_\H(V/L,D)$ of $V/L$. Under this map the constant Darboux
transforms, arising from the linear system $H\subset H^0(V/L)$,
correspond to the trivial representation.

In what follows, we confine our discussion to the case when $M=T^2$ is a
torus. Then more can be said about the structure of the quaternionic spectrum
$\Spec_{\H}(W,D)$ of a holomorphic line bundle $W$. 
Since $T^2 = \R^2/\Gamma$ has abelian fundamental group
$\pi_1(T^2)=\Gamma$, every representation in $\Hom(\Gamma,\H_*)$ can
be conjugated into a complex representation in $\Hom(\Gamma,\C_*)$.
Furthermore, conjugating a complex representation $h$ by the
quaternion $j$ results in the representation $\bar h\in
\Hom(\Gamma,\C_*)$. Therefore, the map
\[
\Hom(\Gamma,\C_*)\to\Hom(\Gamma,\H_*)/\H_*
\]
is $2:1$ away from real representations and the quaternionic spectrum
$\Spec_\H(W,D)$ of $W$ lifts to the {\em complex spectrum}
\[
\Spec(W,D)\subset\Hom(\Gamma,\C_*)\,
\]
given by the possible complex monodromies for  holomorphic sections of
$\tilde{W}$.
By construction the complex spectrum is invariant under complex conjugation
$\rho(h)=\bar{h}$ and the quaternionic spectrum
\begin{equation}
\label{eq:mod-rho}
\Spec_\H(W,D)=\Spec(W,D)/\rho
\end{equation}
is the quotient of the complex spectrum under $\rho$.  

The abelian complex Lie group $\Hom(\Gamma,\C_*)$ has as Lie algebra
$\Hom(\Gamma,\C)$ which, via the period map, is isomorphic to the vector space
of harmonic $1$-forms $\Harm(T^2,\C)$ on the torus $T^2$. The exponential map
\begin{equation}
\label{eq:harm_hom}
\exp\colon \Harm(T^2,\C) \to \Hom(\Gamma,\C_*)\quad\text{where}\quad\exp( \omega)=e^{\int\omega}
\end{equation}
has the integer harmonic forms $\Gamma^* = \Harm(T^2,2\pi i \Z)$ as its kernel
and thus induces the holomorphic isomorphism
\begin{equation}
  \label{eq:isomorphism}
\Harm(T^2, \C)/\Gamma^* \cong \Hom(\Gamma,\C_*)\,.
\end{equation}
In order to see that the spectrum is an analytic variety, we lift the spectrum
to the $\Gamma^*$-periodic {\em logarithmic spectrum}
\begin{equation}
\label{eq:spec_log}
\Specs(W,D)=\exp^{-1}(\Spec(W,D))\subset \Harm(T^2,\C)
\end{equation}
which consists of harmonic forms $\omega$ for which there is a holomorphic
section $\psi\in H^0(\tilde{W})$ with monodromy $h = e^{\int\omega}$.
Interpreting $e^{\int \omega}\in\Hom(\R^2,\C_*)$ as a (non--periodic) gauge
transformation the section $\psi e^{-\int \omega} \in \Gamma(W)$ has trivial
monodromy and lies in the kernel of the gauged operator
\begin{equation}
\label{eq:D_omega_def}
D_{\omega}=e^{-\int \omega} \circ D\circ e^{\int\omega}:\Gamma(W)\to \Gamma(\bar K W)\,.
\end{equation}
The operator $D_{\omega}$ is defined on the torus (even though the gauge is
not) because the Leibniz rule (\ref{eq:hol_structure}) of a quaternionic
holomorphic structure implies
\begin{equation}
\label{eq:D_omega}
 D_\omega(\psi)  = D\psi + (\psi \omega)''\,,
\end{equation}
and both $D$ and $\omega$ are defined on the torus $T^2$. Moreover, $D_\omega$
is elliptic and, due to the term $(\psi \omega)''$ in (\ref{eq:D_omega}),
complex linear (rather than quaternionic linear) between the complex rank 2
bundles $W$ and $\bar KW$ whose complex structures are given by right
multiplication by the quaternion $i$. Thus, the gauge transformation
$e^{\int\omega}$ induces the complex linear isomorphism
\begin{equation}
\label{eq:kerD_omega}
\ker D_{\omega}\to H^{0}_{h}(\tilde{W})\,, \qquad\psi\mapsto \psi
e^{\int\omega}\, ,
\end{equation}
where $H^{0}_{h}(\tilde{W})$ denotes the complex vector space of holomorphic
sections of $W$ with monodromy $h=e^{\int\omega}$. The logarithmic spectrum
therefore is the locus of harmonic forms
\begin{equation}
\label{eq:locus}
\Specs(W,D) = \{ \omega\in\Harm(T^2,\C) \mid \ker D_\omega \neq 0 \}\subset \Harm(T^2,\C)
\end{equation}
for which $D_{\omega}$ has a non--trivial kernel.  But $D_\omega$ is a
holomorphic family of elliptic operators over $\Harm(T^2,\C)$ which implies
that the logarithmic spectrum $\Specs(W,D)$, and hence also the spectrum
$\Spec(W,D)=\Specs(W,D)/\Gamma^{*} $ as a quotient by $\Gamma^*$, are analytic
varieties in $\Harm(T^2,\C)$ and $\Hom(\Gamma,\C_*)$ respectively.

\subsection{Homogeneous bundles}
At this stage it is instructive to discuss an explicit example, namely the
spectrum of a homogeneous torus in $S^4$.  Such a conformally immersed torus
$f\colon T^2 \to S^4$ is a $T^2$--orbit of the M\"obius group and hence a
product of two circles in perpendicular planes in~$S^3$. The induced
holomorphic line bundle $V/L$ is therefore invariant under translations by
$T^2$ and has zero degree. More generally, we call a quaternionic holomorphic
line bundle $W$ over a torus $T^2$ {\em homogeneous} if for each $a\in T^2$
there is a holomorphic bundle isomorphism $T_a\colon W\to a^*W$.  Such bundles
always have degree zero.  What makes it possible to explicitly compute the
spectrum of a homogeneous line bundle is the fact that the holomorphic
structure $D=\delbar+Q$ has constant $Q$ in an appropriate trivialization.
\begin{lemma} 
  Let $W$ be a quaternionic line bundle over a torus $T^2$ with holomorphic
  structure $D=\delbar +Q$. If $Q\neq 0$, then $W$ is homogeneous if and only
  if $W$ is, up to tensoring by a $\Z^2$--bundle, holomorphically isomorphic
  to the trivial $\C^2$--bundle with holomorphic structure
  \[
  \hat D =\begin{pmatrix} \hat\delbar & - \bar{\hat{q}}\\
    \hat{q} & \hat\partial
  \end{pmatrix}
  \]
  where $\hat{q}\in H^0(K)$ is ``constant."  The Willmore integrand of the
  holomorphic structure $D$ is given by $<Q\wedge
  *Q>=i\hat{q}\wedge\bar{\hat{q}}$.
\end{lemma}
\begin{proof}
  Let $J\in\Gamma(\End(W))$ denote the complex structure of $W$. The
  $i$--eigenspace $W_+$ of $J$ is a degree zero complex line bundle with
  holomorphic structure $\delbar$. Since $W_+$ has zero degree $\delbar
  =\delbar_0 -\bar\alpha$ with $\alpha\in H^0(K)$ a holomorphic $1$--form and
  $\delbar_0$ a trivial holomorphic structure. Let $\psi\in \Gamma(W_+)$ be a
  holomorphic trivialization with respect to $\delbar_0$, that is
  $\dbar_0\psi=0$.  In this trivialization 
  \[
  Q\psi=\psi j q
  \]
  for some nonzero $q\in \Gamma(K)$ since $Q$ anticommutes with $J$.  For each
  $a\in T^2$ the holomorphic bundle isomorphism $T_a$ intertwines the complex
  and holomorphic structures on $W$ and $a^*W$, namely
  \[
  (a^*J)T_a=T_aJ \,,\quad (a^*\delbar) T_a = T_a\delbar\,,\quad (a^*Q)T_a =
  T_aQ\,.
  \]
  Evaluating these conditions on the trivializing section $\psi$ gives
  $T_a\psi = (a^*\psi) \lambda_a$ with $\lambda_a\in\C$ and 
  \begin{equation}\label{eq:q-shift}
  a^*q = u_a^2 q\,,
  \end{equation}
  where $u_a$ is the unitary part of $\overline{\lambda_a}$.  This implies
  that $u^2\colon T^2 \to S^1$ is a representation if $q$ is nonzero, and
  hence $u^2 =e^{\int \eta}$ with $\eta\in\Gamma^*$ an integer period harmonic
  form.  Let $R$ be the flat real bundle defined by the representation
  $e^{\frac{1}{2}\int\eta}:\Gamma\to\Z_2$ and denote by $\varphi$ a parallel
  section of $R$ with monodromy $e^{\frac{1}{2}\int\eta}$. Viewing
  $e^{\frac{1}{2}\int\eta}$ as a function on the universal cover $\R^2$, we
  see that $\hat{\psi}=\psi\otimes\phi \,e^{-\frac{1}{2}\int\eta}$ is a
  trivializing section of $W\otimes R$. By construction
  $\hat{\psi}\in\Gamma((W\otimes R)_+)$ and
  \[
  \hat\delbar\hat{\psi}=\hat{\psi}(\bar\alpha-\frac{1}{2}\eta'')=\hat{\psi}\bar{\hat{\alpha}}
  \]
  with $\hat{\alpha}\in H^0(K)$.  Moreover, $\hat{Q}$ is constant in this
  trivialization since
  \[
  \hat{Q}\hat{\psi}=(Q\psi)\otimes\phi \,e^{-\frac{1}{2}\int\eta}=\psi jq  \otimes\phi \,e^{-\frac{1}{2}\int\eta}=   \hat{\psi}j \,e^{-\int\eta}q
  \]
  and by \eqref{eq:q-shift} the $(1,0)$--form $\hat{q}=e^{-\int\eta}q$ is
  translation invariant and thus holomorphic. The form of $\hat{D}$ now
  follows from trivializing $W\otimes R$ using the frame $\psi,\psi j$ and
  \eqref{eq:hol_structure_I}.
  \end{proof}

  We are now in a position to calculate the spectrum of a homogeneous bundle
  $W$: tensoring by the $\Z_2$--bundle given by the representation
  $e^{-\tfrac{1}{2}\int\eta_0}$ for $\eta_0\in\Gamma^*$ relates the spectra of
  $D$ and $\hat{D}$ by this representation. In particular, the logarithmic
  spectra
\[
\Specs(D)=\tfrac{1}{2}\eta_0+\Specs(\hat{D})
\]
are related via a shift by the half lattice vector $\tfrac{1}{2}\eta_0$.
Ignoring this shift for the moment, we may assume that the holomorphic
structure $D$ on $W$ already has constant $Q$ and $D$ is given as in the
lemma.  Moreover, replacing the complex holomorphic structure
$\delbar=\delbar_0-\bar{\alpha}$ in $D=\delbar+Q$, where $\alpha\in H^0(K)$,
by the trivial holomorphic structure $\delbar_0$ results in an additional
shift of $\Specs(D)$ by $\alpha+\bar{\alpha}$ .  Thus, we may assume that
$\delbar$ is already trivial. The holomorphic family of elliptic operators
${D}_{\omega}$, parametrized by $\omega\in \Harm(T^2,\C)$, then has the form
\[
{D}_{\omega} = \begin{pmatrix} {\delbar} + \omega'' & - \bar{ q}\\
     {q} & \partial + \omega'
     \end{pmatrix}
\]
with $q\in H^0(K)$.
Therefore ${D}_\omega$ commutes with translations of the torus $T^2$ so that the
finite dimensional kernel of ${D}_\omega$ is spanned by eigenlines of the
translation operators on $\Gamma(\underline{\C}^2)$. These are given by
the Fourier modes $\psi_\eta = v e^{\int\eta}$ with $v\in\C^2$ and
$\eta\in\Gamma^*$.  From
\begin{equation}
\label{eq:ker_D_omega}
{D}_\omega\psi_\eta = \begin{pmatrix} (\eta +\omega)''& -\bar{ q}\\
     {q} & (\eta +\omega)'
     \end{pmatrix} \psi_\eta
   \end{equation}
we see that $\omega\in \Specs({D})$ if and only if there exists
$\eta\in\Gamma^*$ such that
\begin{equation}
\label{eq:spec_homogeneous}
(\omega+ \eta)'(\omega  +\eta)'' + |q|^2 =0\,.
\end{equation}
We first discuss the vacuum spectrum, that is $Q=0$. In this case the
logarithmic spectrum
 \[
\Specs(\delbar) = \{ \omega\in\Harm(T^2,\C) \, \mid
\,\omega'=\eta'\, \text{ or } \, \omega''=\eta'' \text{ for }
\eta\in\Gamma^*\}
\]
consists of $\Gamma^*$--translates of the complex lines $H^0(K)$ and
$\overline{H^0(K)}$ in $\Harm(T^2,\C)$.  Therefore, the vacuum spectrum
$\Spec(\delbar)\subset \Hom(\Gamma,\C_{*})$ is the union
\begin{equation}\label{eq:vacuumspec}
\Spec(\delbar)=\exp(H^0(K))\cup \exp(\overline{H^0(K)})
\end{equation}
which is a singular curve with double points at the lattice of real
representations. The normalization $\Sigma_0$ of the vacuum spectrum is
disconnected and consists of two copies of $\C$ which are exchanged under the
real structure $\rho$. By adding two points ``at infinity"
$\Sigma_0$ can be compactified to two copies of $\P^1$.  From
\eqref{eq:ker_D_omega} we see that the kernel $\ker \delbar_\omega$ for
$\omega\in\Specs(\delbar)$ is spanned by $(0, e^{\int\eta})$ or
$(e^{\int\eta},0)$, depending on whether $\omega'=-\eta'$ or
$\omega''=-\eta''$. Thus, $\ker \delbar_\omega$ for $\omega
\in\Specs(\delbar)$ is 1--dimensional except for double points of the
spectrum, where it is 2--dimensional.

If $Q\neq 0$ then \eqref{eq:spec_homogeneous} shows that the logarithmic
spectrum $\Specs({D})$ is a $\Gamma^*$--periodic union of conics which are
asymptotic to the logarithmic vacuum spectrum $\Specs(\delbar)$. From
\eqref{eq:spec_homogeneous} we also see that the double point at the trivial
representation in the spectrum $\Spec(\delbar)$ is resolved into a handle in
the spectrum $\Spec({D})$. This has the effect that the normalization $\Sigma$
of the spectrum $\Spec(D)$ is connected with two ends asymptotic to
$\exp(H^0(K))$ and $\exp(\overline{H^0(K)})$ respectively. Thus, $\Sigma$ can
be compactified to $\P^1$ by adding two points at infinity which are
interchanged by the real structure $\rho$. Moreover, from
\eqref{eq:ker_D_omega} we can calculate that the kernel of ${D}_{\omega}$ is
generically one dimensional and approaches the vacuum kernel
$\ker{\delbar}_{\omega}$ near the ends of $\Sigma$.

We conclude this discussion with the case that the homogeneous bundle $W$ has
a 2--dimensional linear system $H\subset H^0(W)$, which is certainly the case
when $W$ is the induced bundle $V/L$ from a homogeneous torus $f:T^2\to S^3$.
Then one can show that the complex holomorphic structure $\hat{\delbar}$ is
necessarily trivial. The half lattice vector $\tfrac{1}{2}\eta_0$ which shifts
the spectrum of $\hat{D}$ to that of the homogeneous torus $f:T^2\to S^3$ is
its induced spin structure.  Since the holomorphic structure on $V/L$ has
sections without monodromy the trivial representation has to be a point of the
spectrum $\Spec(V/L)$.  From \eqref{eq:spec_homogeneous} we see that $|q|^2$
has to be bounded below by the minimum of $|\tfrac{1}{2}\eta_0+\eta|^2$ for
$\eta \in \Gamma^*$. This lower bound on the length of $|q|^2$ translates into
a lower bound on the Willmore energy $\mathcal{W}(f)$ of the homogeneous torus
depending on its conformal and induced spin structures, attaining its minimum
$2\pi^2$ at the Clifford torus.
\begin{figure}[h]
\begin{center}
\begin{minipage}[t]{4cm}
\resizebox{3.5cm}{!}{\includegraphics{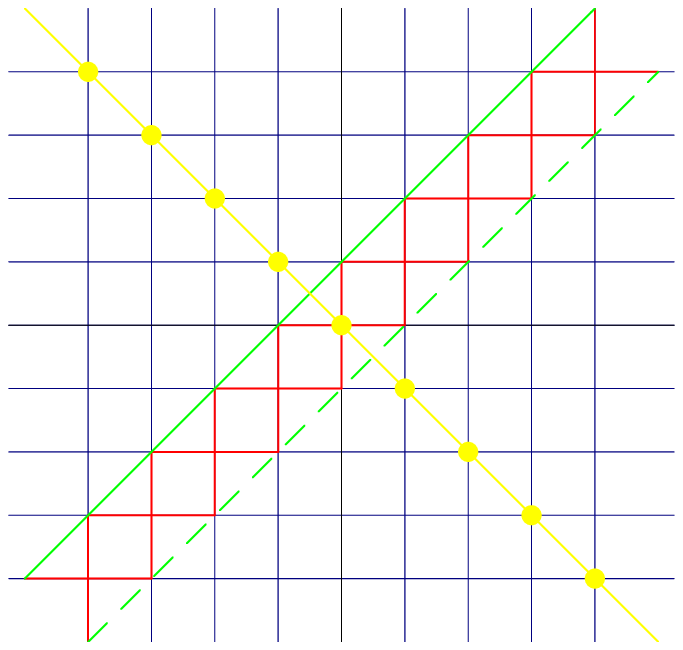}}
\center{$\Specs(W,\delbar)$}
\end{minipage}
\qquad
\begin{minipage}[t]{4cm}
\resizebox{3.5cm}{!}{\includegraphics{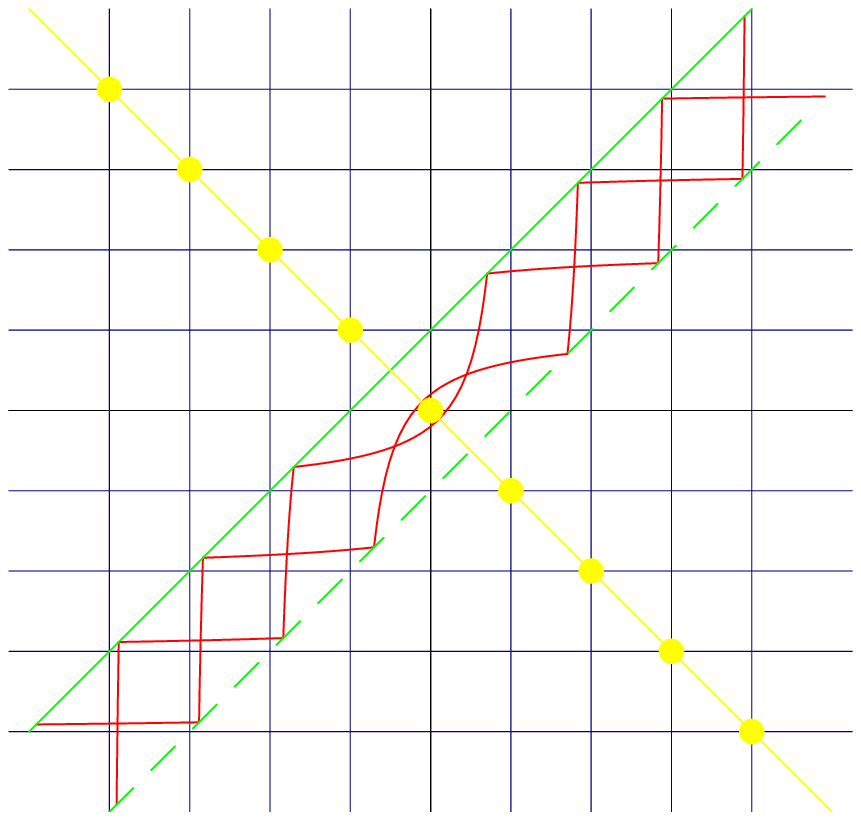}}
\center{$0\not\in\Specs(W,\delbar +Q)$}
\end{minipage}
\qquad
\begin{minipage}[t]{4cm}
\resizebox{3.5cm}{!}{\includegraphics{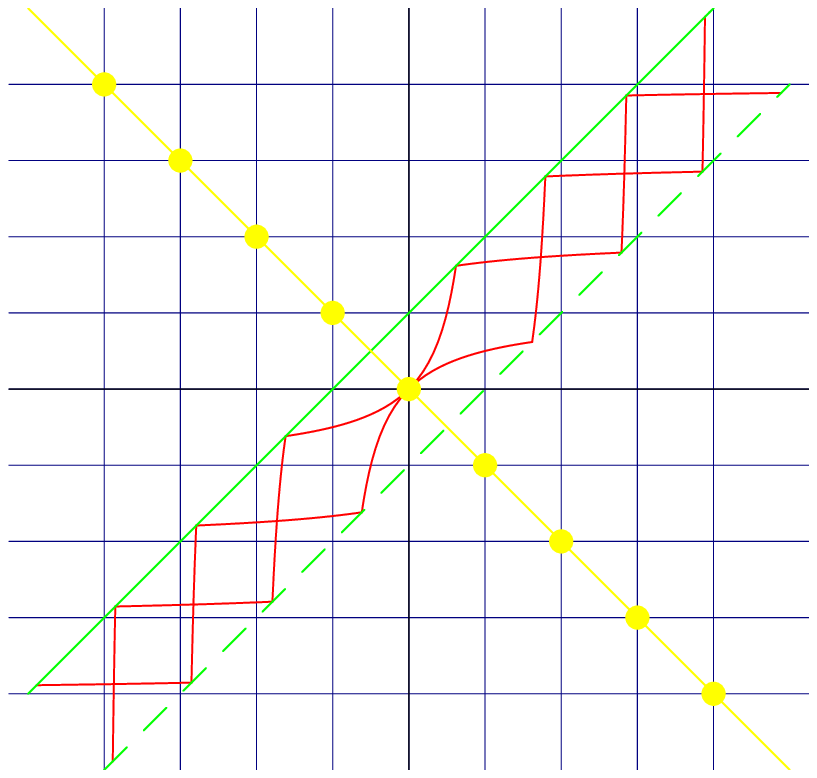}}
\center{$0\in\Specs(W,\delbar +Q)$}
\end{minipage}
\end{center}
\end{figure}
\begin{figure}[h]
\begin{center}
\begin{minipage}[t]{3.5cm}
\resizebox{3.5cm}{!}{\includegraphics{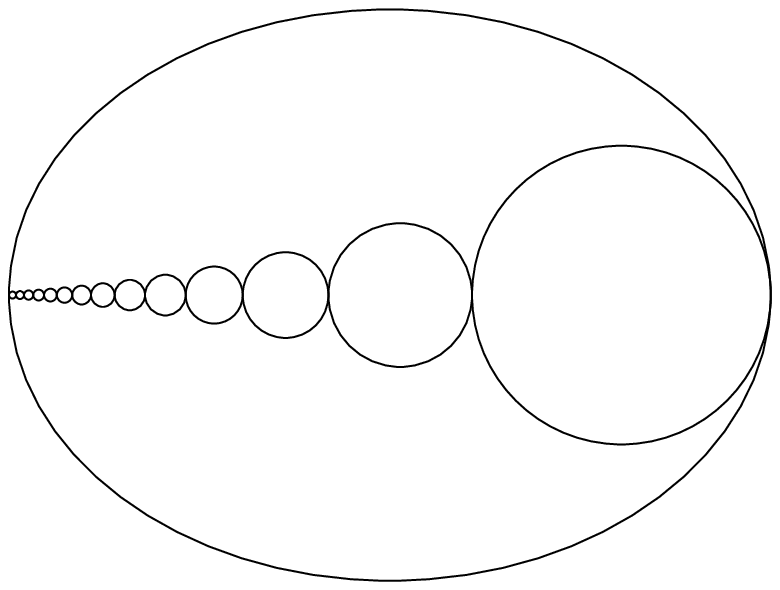}}
\center{$\Spec(W,\delbar)$}
\end{minipage}
\qquad
\begin{minipage}[t]{3.5cm}
\resizebox{3.5cm}{!}{\includegraphics{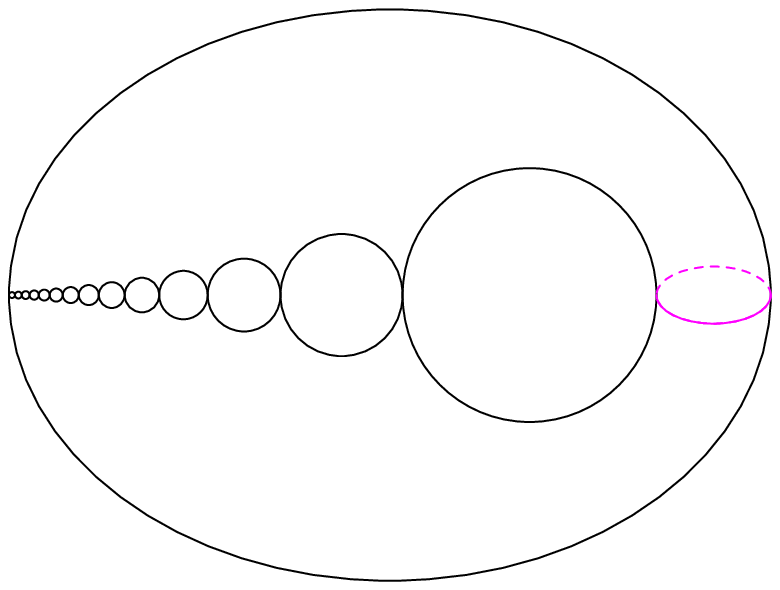}}
\center{$\Spec(W,\delbar+Q)$}
\end{minipage}
\caption{Vacuum spectrum and spectrum of a homogeneous torus.}
\label{fig:tangent0}
\end{center}
\end{figure}
\subsection{The spectral curve and the kernel bundle}
The behavior of the spectrum of a homogeneous line bundle is in many ways
reminiscent of the general case. The following theorem collects some of the
results about the structure of the spectrum of a quaternionic holomorphic line
bundle needed in the present paper and proven in \cite{ana}. We may assume the
line bundle to have zero degree since otherwise its spectrum is finite or all
of $\Hom(\Gamma,\C_{*})$.
\begin{theorem}\label{thm:main1} 
Let $W$ be a quaternionic line bundle of
degree zero over a torus $T^2$ with holomorphic structure $D=\delbar +Q$. 
\begin{enumerate}
\item The spectrum $\Spec(W,D)$ is a $1$--dimensional analytic variety in
  $\Hom(\Gamma,\C_{*})$ invariant under the real structure $\rho(h)=\bar{h}$.
\item The spectrum $\Spec(W,D)$ is asymptotic to the vacuum spectrum
  $\Spec(W,\delbar)$: outside a sufficiently large compact set in
  $\Hom(\Gamma,\C_{*})$ the spectrum $\Spec(W,D)$ is contained in an
  arbitrarily small tube around $\Spec(W,\delbar)$. Away from the double
  points of the vacuum spectrum outside this compact set the spectrum is a
  graph over the vacuum spectrum. Near the double points of the vacuum
  spectrum outside this compact set the spectrum is a smooth annulus or two
  discs intersecting in a double point.
\item For $h=e^{\int\omega}\in\Spec(W,D)$ the operator $D_{\omega}$
  generically has a $1$--dimensional kernel, that is, $\dim
  H^0_{h}(\tilde{W})=1$ away from a discrete set in $\Spec(W,D)$.
\end{enumerate}
\end{theorem}
Normalizing the spectrum we obtain a Riemann surface, the {\em spectral curve}
of the quaternionic holomorphic line bundle.

\begin{definition}
  \label{def:spectral_curve} Let $W$ be a quaternionic line bundle of degree
  zero over a torus $T^2$ with holomorphic structure $D=\delbar +Q$. The {\em
    spectral curve} $\Sigma$ of $W$ is the normalization
  $h\colon\Sigma\to\Spec(W,D)$ of the spectrum $\Spec(W,D)$ to a (not
  necessarily connected) Riemann surface.

The {\em logarithmic spectral curve} is the Riemann surface $\tilde{\Sigma}$
normalizing the logarithmic spectrum $\omega\colon
\tilde{\Sigma}\to\Specs(W,D)$. By the universal property of the normalization
$\tilde{\Sigma}$ is a $\Gamma^*$-periodic Riemann surface whose quotient under
the lattice $\Gamma^*$ is $\Sigma=\tilde{\Sigma}/\Gamma^*$.

The {\em vacuum spectral curve} $\Sigma_0$ is the Riemann surface which
normalizes the spectrum $\Spec(W,\delbar)$ of the complex holomorphic
structure $\delbar$.
\end{definition}

The normalization map $\omega\colon \tilde{\Sigma}\to \Specs(W,D)$ in
Definition~\ref{def:spectral_curve} pulls back the holomorphic family
$D_{\omega}\colon \Gamma(W)\to\Gamma(\bar{K}W)$ to a holomorphic family of
elliptic operators over the Riemann surface $\tilde{\Sigma}$.  Thus, by
Theorem~\ref{thm:main1} the kernels of $D_{\omega}$ give rise to a holomorphic
line bundle
\[
\tilde{\mathcal{L}}\to\tilde{\Sigma}\,,
\]
a subbundle of the trivial $\Gamma(W)$--bundle over $\tilde{\Sigma}$: the
fibers $\tilde{\mathcal{L}}_{\tilde{x}}= \ker D_{\omega(\tilde{x})}$ coincide
with the kernels of $D_{\omega(\tilde{x})}$ away from the
discrete set of points $\tilde{x}\in\tilde{\Sigma}$ where the kernel dimension
is greater then one. Moreover, on this discrete set
$\tilde{\mathcal{L}}_{\tilde{x}}\subset \ker D_{\omega(\tilde{x})}$.

We have seen in \eqref{eq:kerD_omega} that $\ker
D_{\omega}$ is isomorphic to $H^0_{h}(\tilde{W})$ where $h=e^{\int \omega}$.
Since the spectral curve $\Sigma=\tilde{\Sigma}/\Gamma^*$ is the quotient
under the action of the lattice $\Gamma^*$ of integer harmonic forms, the
bundle $\tilde{\mathcal{L}}$ descends to a holomorphic line bundle
\begin{equation}
\label{eq:kernel-bundle}
\mathcal{L}\to\Sigma\,,\qquad \mathcal{L}_x=\tilde{\mathcal{L}}_{\tilde{x}} e^{\int\omega(\tilde{x})}
\end{equation}
over $\Sigma$, where $x=\tilde{x}+\Gamma^*$.  Therefore, the fiber
$\mathcal{L}_x\subset H^0_{h(x)}(\tilde{W})$ of $\mathcal{L}$ over
$x\in\Sigma$ consists of holomorphic sections of $W$ with monodromy $h(x)$ and
$\mathcal{L}_x= H^0_{h(x)}(\tilde{W})$ away from a discrete set 
in $\Sigma$.  
\begin{definition}\label{def:kernel_bundle}
  Let $W$ be a quaternionic holomorphic line bundle of degree zero over a
  torus $T^2$ with spectral curve $\Sigma$. The complex holomorphic line
  bundle $\mathcal{L}$ over $\Sigma$ given in \eqref{eq:kernel-bundle} with
  generic fiber $\mathcal{L}_x= H^0_{h(x)}(\tilde{W})$ over $x\in\Sigma$ is
  called the {\em kernel bundle}.
\end{definition}
The real structure $\rho$ on $\Sigma$ acts on the kernel bundle via
multiplication by the quaternion $j$: given a holomorphic section $\psi\in
H^0_{h(x)}(\tilde{W})$ the section $\psi j$ has monodromy
$\overline{h(x)}=h(\rho(x))$.  Therefore, the real structure $\rho$ is fixed
point free and the kernel bundle is compatible with the real structure, that
is, $\rho^*\mathcal{L}=\mathcal{L}j$.
In particular,  $\Sigma/\rho$
normalizes $\Spec_{\H}(W,D)$ to a non--orientable (if $\Sigma$ is connected)
Riemann surface by \eqref{eq:mod-rho}.

From Theorem~\ref{thm:main1} and the fact that $\Sigma$ cannot have compact
components, we see that $\Sigma$ has one or two ends and at most two
components each of which contains an end and which are exchanged under $\rho$.
In case $\Sigma$ has one end $\Sigma$ is connected and the genus of $\Sigma$
is necessarily infinite. In the finite genus case $\Sigma$ has two ends. We
summarize the discussion so far:
\begin{theorem}\label{thm:analysis}
  Let $W$ be a quaternionic holomorphic line bundle of degree zero over a
  torus~$T^2$. Then the spectral curve $\Sigma$ is a Riemann surface with a
  fixed point free real structure~$\rho$, one or two ends and at most two
  components each of which contains an end.  Depending on whether $\Sigma$ has
  one or two ends the genus of $\Sigma$ is infinite and $\Sigma$ is connected,
  or the genus of $\Sigma$ is finite.
The kernel bundle
 $\mathcal{L}$ is a complex holomorphic line bundle over $\Sigma$ compatible
 with the real structure, that is, $\rho^*\mathcal{L}=\mathcal{L}j$.
\end{theorem}

We now return to the case of interest to us when the quaternionic holomorphic
line bundle $W$ is the line bundle $V/L$ induced by a conformal immersion
$f\colon T^2\to S^4$ of zero normal bundle degree. In this situation the
fibers of the kernel bundle $\mathcal{L}$ have a geometric interpretation in
terms of Darboux transformations.  By Definition~\ref{def:kernel_bundle} the
fiber $\mathcal{L}_x\subset H^0_{h(x)}(\widetilde{V/L})$ of $\mathcal{L}$ over
$x\in\Sigma$ consists of a complex line worth of non--trivial sections $\psi$
of $\widetilde{V/L}$ with monodromy $h(x) \in\Hom(\Gamma,\C_{*})$.  As we have
seen in Lemma~\ref{l:prolong} such $\psi$ can be prolonged to sections
$\hat{\psi}\in\Gamma_{h(x)}(\tilde{V})$ with the same monodromy $h(x)$.
Therefore, we have constructed a complex holomorphic line subbundle
\[
\hat{\mathcal{L}}\to\Sigma
\]
of the trivial $\Gamma(\tilde{V})$--bundle over the spectral curve $\Sigma$.
Moreover, Lemma~\ref{l:singular_darboux} shows that for $x\in\Sigma$ the
quaternionic line subbundle $L^x=\hat{\psi}\H$ of the trivial $\H^2$--bundle
$V$ over the torus $T^2$ with $\hat{\psi}\in \hat{\mathcal{L}}_x$ is a
(possibly singular) Darboux transform
\[
f^x\colon T^2\to S^4
\]
of $f$.  On the other hand, evaluating the holomorphic line subbundle
$\hat{\mathcal{L}}$ of the trivial $\Gamma(\tilde{V})$--bundle for fixed $p\in
T^2$ gives the holomorphic map
\[
\hat{F}(p,-)\colon\Sigma\to \P^3\quad\text{where}\quad \hat{F}(p,-)=\hat{\mathcal{L}}(p)\subset \H^2\,.
\]
From the invariance of the kernel bundle we see that
$\rho^*\hat{F}(p,-)=\hat{F}(p,-)j$. This implies that under twistor projection
$\C\P^3\to\H\P^1$ the map $\hat{F}(p,-)$ induces the conformal map
\[
F(p,-)\colon \Sigma\to\H\P^1\,,\quad F(p,x)=f^x(p)
\]
realizing the spectral curve $\Sigma$, in fact its quotient $\Sigma/\rho$
under the real structure $\rho$, as the twistor projection into $S^4$ of a
holomorphic curve in $\P^3$, that is, as a super conformal Willmore surface.

Generally, the $T^2$--family $\hat{F}(p,-)$ of holomorphic curves in $\P^3$
will not be smooth in~$p\in T^2$: second and higher order zeros on $T^2$ of
sections $\psi\in\mathcal{L}_x$ give rise to zeros of the prolonged section
$\hat{\psi}\in\Gamma(\tilde{V})$, which can cause bubbling off phenomena.

The situation simplifies considerably under the assumption that the conformal
immersion $f\colon T^2\to S^4$ has $\mathcal{W}(f)<8\pi$. Then
Lemma~\ref{l:emb_Darboux} ensures that the sections in $\hat{\mathcal{L}}_x$
for $x\in\Sigma$ have no zeros and every non--constant Darboux transform $f^x$
is a conformal embedding.
\begin{theorem}\label{thm:F-map}
  Let $f\colon T^2\to S^4$ be a conformal immersion with trivial normal
  bundle, induced holomorphic line bundle $V/L$, and spectral curve $\Sigma$.
  Then there exists a map
\[
F\colon T^2\times \Sigma\to S^4
\]
with the following properties:
\begin{enumerate}
\item For $x\in \Sigma$ the map $f^x=F(-,x)\colon T^2\to S^4$ is a (possibly
  singular) Darboux transform of $f$. In the non-singular case its Willmore
  energy is $\mathcal{W}(f^x)=\mathcal{W}(f)$. If the Willmore energy
  $\W(f)<8\pi$ then each non--constant Darboux transform $f^x$ is a conformal
  embedding.
\item For $p\in T^2$ the map $F(p,-)\colon \Sigma\to S^4$ is the twistor
  projection of a holomorphic curve $\hat{F}(p,-)\colon \Sigma\to \P^3$
  equivariant with respect to the real structure $\rho$ on $\Sigma$ and
  multiplication by $j$ on $\P^3=\P(\H^2)$.
\item If the Willmore energy of $f$ satisfies $\mathcal{W}(f)<8\pi$, then
  $\hat{F}\colon T^2\times\Sigma\to \P^3 $ is a smooth map which is conformal
  in the first and holomorphic in the second factor.
\end{enumerate}
\end{theorem}

We conclude this section with a fact which lies at the heart of the integrable
systems approach to conformal maps of $2$--tori into the $4$--sphere: the
spectral curve is a first integral of the discrete flow given by Darboux
transforms on the space of conformal tori in $S^4$ with zero normal bundle
degree.
\begin{theorem}
  Let $f,f^{\sharp}\colon T^2\to S^4$ be conformal immersions so that
  $f^{\sharp}$ is a Darboux transform of $f$. Then the spectral curves
 of $f$ and $f^{\sharp}$ agree, that is,  $\Sigma=\Sigma^\sharp$.
\end{theorem}
\begin{proof}

From Remark~\ref{rem:bianchi} we know that $\Spec(V/L)\subset
\Spec(V/L^{\sharp})$. But the spectrum is an analytic curve in
$\Harm(T^2,\C_*)$ asymptotic to the vacuum spectrum by
Theorem~\ref{thm:main1}. This implies that the spectra of $V/L$ and
$V/L^{\sharp}$ have to agree.
\end{proof}

\section{Conformal $2$--tori in $S^4$ of finite spectral genus}
At this stage we have constructed to each conformal immersion $f\colon T^2\to
S^4$ of a $2$--torus into the 4--sphere with trivial normal bundle a Riemann
surface, the spectral curve ~$\Sigma$, of possibly infinite genus, with one or
two ends and a fixed point free real structure $\rho$. This curve carries a
complex holomorphic line bundle $\hat{\mathcal{L}}\to \Sigma$, the
prolongation of the kernel bundle~$ \mathcal{L}$, which gives rise to Darboux
transforms of $f$.  The line bundle is compatible with the real structure in
the sense that $\rho^*\hat{\mathcal{L}}=\hat{\mathcal{L}}j$.  For these
``spectral data'' to be of algebro--geometric nature, the curve $\Sigma$ has
to complete to a compact Riemann surface and the line bundle
$\hat{\mathcal{L}}$ has to extend holomorphically to the compactification.
For the generic conformal immersion $f$ this will not be possible. But there
are interesting examples of conformal immersions, including constant mean
curvature \cite{PS89,Hi} and (constrained) Willmore tori \cite{S02,cw_tori},
for which the spectral data become algebro--geometric due to the ellipticity
of the underlying variational problems.  From the asymptotic behavior of the
spectral curve $\Sigma$ outlined in Theorem~\ref{thm:main1}, we know that away
from handles each end of $\Sigma$ is asymptotic to a plane.  Thus, if $\Sigma$
has finite genus, and therefore must have two ends, it can be compactified by
adding two points $o$ and $\rho(o)=\infty$ at the ends.
\begin{definition}
  A conformal immersion $f\colon T^2\to S^4$ has {\em finite spectral genus}
  if its spectral curve $\Sigma$ has finite genus. In this case $\Sigma$ has
  two ends and can be completed to a compact Riemann surface $\bar{\Sigma}$ by
  adding two points $o$ and $\rho(o)=\infty$ at the ends.
\end{definition}
In the mathematical physics literature also the term ``finite gap" is used:
asymptotically, $\Spec(V/L,D)$ with $D=\delbar+Q$ lies near the vacuum spectrum
$\Spec(V/L,\delbar)$ which has a lattice of double points. Generically
$\Spec(V/L,D)$ resolves all the double points into handles, but in the finite
spectral genus case only finitely many handles appear, and the ends can be
compactified by adding two points.

We now show that the $T^2$--family of holomorphic curves from $\Sigma$ to
$\P^3$ given by the prolonged kernel bundle $\hat{\mathcal{L}}$ extends
holomorphically to the compactified spectral curve
$\bar{\Sigma}=\Sigma\cup\{o,\infty\}$.
\begin{theorem}
  Let $f\colon T^2\to S^4$ be a conformal immersion with trivial normal bundle
  whose spectral curve $\Sigma$ has finite genus.  As in
  Theorem~\ref{thm:F-map} we denote by $\hat{F}\colon T^2\times \Sigma\to
  \P^3$ the $T^2$--family of holomorphic curves with twistor projection
  $F\colon T^2\times\Sigma\to S^4$.  Then $\hat{F}$ extends to the
  compactification $\hat{F}\colon T^2\times \bar{\Sigma}\to\P^3$ and
  satisfies:
\begin{enumerate}
\item
$\hat{F}(p,-)\colon \bar{\Sigma}\to \P^3$ is  an algebraic curve for each $p\in T^2$.
\item
The original conformal immersion $f\colon T^2\to S^4$ is obtained by evaluation 
\[
f=F(-,o)=F(-,\infty)
\]
of the $T^2$--family at the points at infinity.
\item If the Willmore energy satisfies $\mathcal{W}(f)<8\pi$ then the map
  $\hat{F}\colon T^2\times\bar{\Sigma}\to \P^3$ is a smooth $T^2$--family of
  algebraic curves.
\end{enumerate}
\end{theorem}
The proof  of this theorem is based on the following result from \cite{ana}: 
\begin{lemma}
  \label{lem:ana2} Let $W$ be a quaternionic line bundle of degree zero over a
  torus $T^2$ with complex structure $J$ and holomorphic structure $D$ whose
  spectral curve $\Sigma$ has finite genus.
\begin{enumerate}
\item Outside a sufficiently large compact set in $\Sigma$ the fibers
  $\mathcal{L}_x$ of the eigenline bundle $\mathcal{L}$ are spanned by a
  holomorphic family of nowhere vanishing holomorphic sections $\psi^x\in
  H^0_{h(x)}(\tilde{W})$.
\item Outside a sufficiently large compact set $K\subset \Sigma$ the
  holomorphic family of complex structures $S^x$ on $W$ given by
  $S^x\psi^x=\psi^x i$ extends holomorphically to $\bar{\Sigma}\setminus K$
  with $S^o= J$ and $S^{\infty}=- J$.
\end{enumerate}
\end{lemma}
Applying this lemma to the bundle $V/L$ induced by a conformal immersion is a
key ingredient of the proof of the above theorem.
\begin{proof}
  We have to show that for given $p\in T^2$ the holomorphic curve
  $\hat{F}(p,-)\colon \Sigma\to\P^3$ extends to the compactification
  $\bar{\Sigma}$. Let $L_0$ be a fixed Darboux transform of $f$ such that
  $V=L\oplus L_0$.  From Lemma~\ref{lem:ana2} we know that for $x\in\Sigma$
  near an end the fiber $\hat{\mathcal{L}}_x=\hat\psi^x\C$ of the prolonged
  kernel bundle $\hat{\mathcal{L}}$ is spanned by the prolongation
  $\hat{\psi^x}$ of a nowhere vanishing holomorphic section $\psi^x\in
  H^0_{h(x)}(\widetilde{V/L})$ with monodromy $h(x)$. Thus, we have a Darboux
  transform $L^x=\hat\psi^x\H\subset V$ for such $x\in\Sigma$. Since $\psi^x$
  is holomorphic the connection $\nabla^x$ on $V/L$ rendering $\psi^x$
  parallel is adapted, that is, $(\nabla^x)''=D$ is the holomorphic structure
  induced by $f$ on $V/L$.  Moreover, by Lemma~\ref{lem:ana2} the holomorphic
  family of complex structures $S^x$ given by $S^x\psi^x=\psi^x i$ extends
  holomorphically to the compactification $\bar{\Sigma}$ with $S^{\infty}=J$
  and $S^o=-J$, where $J$ is the complex structure \eqref{eq:conformal} on
  $V/L$ coming from the immersion $f$.

  We express $L^x=(1+R^x)L_0$ as a graph where $R^x\in\Gamma(\Hom(V/L,L))$.
  Then the adapted connection $\nabla^x$ is given by
\[
\nabla^x=\nabla_0+\delta R^x\,,
\]
where $\nabla_0$ is the adapted connection on $V/L$ induced by the splitting
$V=L\oplus L_0$ from Corollary~\ref{c:flat_adapted}. Denote by $\hat{\nabla}$
the unique flat connection on $V/L=E_0\oplus E_0$ compatible with $J$,
$\delbar$, and with unitary monodromy, where $E_0$ is the $i$--eigenspace of
$J$ on $V/L$. Then
\[
\nabla^x=\hat{\nabla}+\eta^x\quad\text{and}\quad \nabla_0=\hat{\nabla}+\eta_0
\]
and we have
\begin{equation}\label{eq:R}
\eta^x=\eta_0+\delta R^x\,.
\end{equation}
Now let $E^x\subset V/L\cong L_0$ denote the $i$--eigenspace of $S^x$.  Then
$(1+R^x)E^x=\hat{\psi^x}\C\subset L^x$ is the complex line bundle over $T^2$
spanned by $\hat{\psi^x}\in\hat{\mathcal{L}}_x$ over $x\in\Sigma$. We thus
have to show that as $x$ tends to $\infty$ the line bundle $(1+R^x)E^x$
approaches the twistor lift $E\subset L$ which is the $i$--eigenspace of the
complex structure $\tilde{J}$ on $L$ induced \eqref{eq:conformal} by the
conformal immersion $f$.  Note that since $\delta\in\Gamma(K\End_+(L,V/L))$ we
have $\delta E=KE_0$.

From $S^{\infty}=J$ we see that $S^x(p)\neq -J(p)$ for all $p\in T^2$ provided
$x$ is in a sufficiently small punctured neighborhood of
$\infty\in\bar{\Sigma}$. Thus, stereographic projection gives
\[
S^x=(1+Y^x)J(1+Y^x)^{-1}
\]
for a holomorphic family $Y^x\in\Gamma(\End_{-}(V/L))$ on that punctured
neighborhood.  The holomorphic family of flat adapted connections
\[
\nabla^x=(1+Y^x)\circ(\hat{\nabla}+\alpha^x)\circ(1+Y^x)^{-1}
\]
gauges under $1+Y^x$ to the holomorphic family of flat connections
$\hat{\nabla}+\alpha^x$ so that $\alpha^x$ is a $\hat{\nabla}$--closed form in
$\Omega^1(\End_+(V/L))$ .  Then 
\[
\eta^x=\frac{1}{1+|Y^x|^2}(\alpha^x +O(1))
\]
and it is shown in  (5.6) of \cite{ana} that
\[
\alpha^x=\sum_{k\geq -1}\alpha_k x^{-k}
\]
has a simple pole at $\infty$ with residue $\alpha_{-1}\in H^0(K)$ 
a non--zero holomorphic $1$--form on $T^2$. 
Observing that $E^x=(1+Y^x)E_0$,
we therefore obtain from \eqref{eq:R} that
\[
(1+R^x)E^x=(1+\delta^{-1}(\eta^x-\eta_0))(1+Y^x)E_0\,.
\]
But this last expression tends to 
\[
E=\delta^{-1}KE_0\subset L
\]
as $x$ tends to $\infty$, where we used that the only term which is not
bounded is given by the simple pole of $\alpha^x$ at $\infty$.
\end{proof}
It is precisely this theorem which will allow us to study conformal tori in
the $4$--sphere of finite spectral genus by algebro--geometric techniques. In
\cite{ana} it is shown that the Willmore energy of $f$ can be computed in
terms of residues of appropriate meromorphic forms on the compact curve
~$\bar{\Sigma}$. Moreover, the $T^2$--family of algebraic curves is in fact
linear as a map into the Jacobian of $\bar{\Sigma}$. But much more is true:
the Jacobian acts via flows, the Davey--Stewartson hierarchy, preserving the
Willmore energy of the conformal torus $f\colon T^2\to S^4$, and $f$ itself is
the flow tangent to the Abel image of $\bar{\Sigma}$ in its Jacobian.
\begin{appendix}
\section{}
\label{appendix}
In this appendix we summarize the basic notions concerning the theory
of quaternionic vector bundles over Riemann surfaces \cite{Klassiker}.
A quaternionic vector bundle $W$ with complex structure $J$ over a
Riemann surface $M$ decomposes into 
\begin{equation}
\label{eq:splitting}
W = W_+ \oplus W_-\,,
\end{equation}
 where
$W_\pm$ are the $\pm i$--eigenspaces of $J$. By restriction $J$
induces complex structures on $W_\pm$ and $W_- = W_+ j$ gives a
complex linear isomorphism between $W_+$ and $W_-$.  The degree of the
quaternionic bundle $W$ with complex structure $J$ is then defined to be
the degree of the underlying complex vector bundle
\begin{equation}
\label{eq:degree}
\deg W := \deg W_+\,,
\end{equation}
which is half of the usual degree of $W$ when viewed as a complex
bundle with $J$.

Given two quaternionic bundles $W$ and $\hat{W}$ with complex
structures $J$ and $\hat{J}$ the complex linear homomorphisms
$\Hom_+(W,\hat{W})$ are complex linearly isomorphic to
$\Hom_\C(W_+,\hat{W}_{+})$. On the other hand, the complex antilinear
homomorphisms $\Hom_-(W, \hat{W})$ are complex linearly isomorphic to
$\Hom_+(\bar{W}, \hat{W})$, where the complex structure on a
homomorphism bundle is induced by the target complex structure.

A {\em quaternionic holomorphic} structure on a quaternionic vector
bundle $W$ with complex structure $J$ is given by a quaternionic
linear operator
\begin{equation}
\label{eq:hol_structure}
D\colon \Gamma(W)\to\Gamma(\bar{K}W)
\end{equation}
satisfying the Leibniz rule $D(\psi\lambda) = (D\psi)\lambda + (\psi
d\lambda)''$, where $\psi\in\Gamma(W)$ and $\lambda\colon M \to \H$.  Note
that $(d\lambda)''$ cannot be defined since $\H$ does not have a preferred
complex structure, whereas $(\psi d\lambda)''$ is well--defined as the
$\bar{K}$--part of a $W$--valued 1--form.  The decomposition of $D$ into $J$
commuting and anticommuting parts gives
\begin{equation}
\label{eq:hol_structure_decomp}
D = \delbar  + Q\,.
\end{equation}
Here $\delbar$ is the double of a complex holomorphic structure on the complex
bundle $W_+$ and $Q\in\Gamma(\bar{K}\End_{-}(W))$ is a $1$--form of type
$\bar{K}$ with values in complex antilinear endomorphisms of $W$.

The quaternionic vector space of holomorphic sections of $W$ is
denoted by
\begin{equation}
\label{eq:hol_sections}
H^{0}(W)= \ker (D)
\end{equation}
and has finite dimension $h^0(W):=\dim H^0(W)$
for compact $M$.  The $L^2$--norm
\begin{equation}
\label{eq:Willmore_energy_holbundle}
\mathcal{W}(W)= \mathcal{W}(W, D) = 2\int_M <Q\wedge *Q>
\end{equation}
of $Q$ is called the {\em Willmore energy} of the holomorphic bundle $W$,
where $<\,,\,>$ denotes the trace pairing on $\End(W)$. The special case
$Q=0$, for which $\mathcal{W}(W)=0$, describes (doubles of) complex
holomorphic bundles $W=W_+\oplus W_+$.  A typical example of a quaternionic
holomorphic structure arises from the $\bar{K}$--part $\nabla''$ of a
quaternionic connection $\nabla$ on $W$.

For analytic considerations, and to make contact to the theory of Dirac
operators with potentials, it is common to regard $W$ as a complex vector
bundle whose complex structure $I$ is given by right multiplication by the
quaternion $i$.  Thus, $W = W_+ \oplus \overline{W}_+$ so that $Q$ is
$I$--linear and $D = \delbar+ Q$ has the decomposition
\begin{equation}
\label{eq:hol_structure_I}
D = \begin{pmatrix} \delbar & 0\\
     0 & \partial 
     \end{pmatrix} + \begin{pmatrix} 0 &-\bar{q} \\  q& 0
    \end{pmatrix}
\end{equation}  
with $q\in\Gamma(K\Hom(W_+,\overline{W}_+))$.

Given a \emph{linear system} of a holomorphic bundle $W$, that is, a
linear subspace $H\subset H^0(W)$, we have the Kodaira map 
\begin{equation}
\label{eq:Kodaira}
\ev^*\colon  W^*\to H^*\,.
\end{equation}
Here $\ev$ denotes the bundle map which, for $p\in M$, evaluates holomorphic
sections at~$p$. In case the Kodaira map is injective, its image $W^*\subset
H^*$ defines a map into the Grassmannian $f\colon M \to \text{Gr}(r, H^*)$,
where $\rank W = r$.  The derivative $\delta=\pi d|_{W^*}$ of this map
satisfies $*\delta= \delta J$, where $d$ denotes the trivial connection on
$H^*$ and $\pi\colon H^* \to H^*/W^*$ denotes the canonical projection. In
other words, $f$ is a holomorphic curve \cite{Klassiker}.

If $M$ is a compact Riemann surface of genus $g$ and $L$ is a holomorphic line
bundle over $M$, the Pl\"ucker formula \cite{Klassiker} gives a lower bound
for the Willmore energy $\W(L)$ of $L$ in terms of the genus $g$ of $M$, the
degree of $L$, and vanishing orders of holomorphic sections of $L$.  For the
purposes of this paper, we need a more general version of the Pl\"ucker
formula which includes linear systems with monodromy.  If $\tilde M$ is the
universal cover of $M$, we denote by $\tilde L$ the pullback to $\tilde M$ of
$L$ by the covering map. The holomorphic structure of $L$ lifts to make
$\tilde L$ into a holomorphic line bundle.

A \emph{linear system with monodromy} is a linear subspace $H\subset
H^0(\tilde L)$ of holomorphic sections of $\tilde L$ such that
\begin{equation}
\label{eq:system_monodromy}
\gamma^*H=H \quad \text{ for all } \quad \gamma\in\pi_1(M)\,,
\end{equation}
where $\pi_1(M)$ acts via deck transformations.  Adapting the proof in
\cite{Klassiker} by replacing the trivial connection with a flat
connection on $H^*$, we obtain the Pl\"ucker formula for an
$n$--dimensional linear system $H$ with monodromy
 \begin{equation}
\label{eq:Pluecker}
\tfrac{1}{4\pi}\W(L)\ge n((n-1)(1-g) - \deg L) + \ord H\,.  
\end{equation}
The order $\ord H$ of the linear system $H\subset H^0(\tilde L)$ with
monodromy is computed as follows: to a point $x\in\tilde M$ we assign
the Weierstrass gap sequence $n_0(x)<\ldots <n_{n-1}(x)$ inductively
by letting $n_k(x)$ be the minimal vanishing order strictly greater
than $n_{k-1}(x)$ of holomorphic sections in $H$. Away from isolated
points this sequence is $n_k(x)=k$ and
\[
\ord_x H = \sum_{k=0}^{n-1}( n_k - k )
\]
measures the deviation from the generic sequence. 
Since $H$ is a linear system with mo\-no\-dro\-my the Weierstrass gap
sequence is invariant under deck transformations. Therefore, $\ord_pH$
for $p\in M$ is well--defined, zero away from finitely many
points, and  the order of the linear system $H$ is given by
\begin{equation}
\label{eq:ord_H}
\ord H = \sum_{p\in M}\ord_pH\,.
\end{equation}
For a linear system $H\subset H^0(L)$ without monodromy the formula
(\ref{eq:Pluecker}) is the usual Pl\"ucker relation as discussed in
\cite{Klassiker}.
\end{appendix}

\end{document}